\documentclass[11pt,a4paper]{article}
\usepackage[latin1]{inputenc}
\usepackage[colorlinks]{hyperref}
\usepackage{amsmath,amsthm,amsfonts,amssymb}
\usepackage{enumerate}
\usepackage{array}
\usepackage{todonotes}
\usepackage{mathtools}
\usepackage{url}
\numberwithin{equation}{section}
\theoremstyle{plain}
\newtheorem{theorem}{Theorem}[section]
\newtheorem{lemma}[theorem]{Lemma}
\newtheorem{proposition}[theorem]{Proposition}

\newtheorem*{conjecture*}{Conjecture}
\newtheorem*{lemma*}{Lemma}
\theoremstyle{definition}

\newtheorem*{remark*}{Remark}
\newcommand{\mailto}[1]{\href{mailto:#1}{\nolinkurl{#1}}}

\newlength{\hdelta}
\newlength{\vdelta}
\newcommand{\R}{\mathbb{R}}
\newcommand{\Rp}{{\R_+}}

\newcommand{\E}{\operatorname{E}}
\newcommand{\pr}{\operatorname{P}}

\newcommand{\cov}{\operatorname{Cov}}
\newcommand{\var}{\operatorname{Var}}

\newcommand{\I}{\mathbf{1}}
\newcommand{\ind}{\I}

\newcommand{\matern}{Mat\'{e}rn}

\newcommand{\lambdath}{\lambda_{\rm th}} 

\newcommand{\Phith}{\Phi^{\rm th}} 
\newcommand{\Xth}{X^{\rm th}}      
\newcommand{\pth}{p_{\rm th}}      
\newcommand{\Fth}{F_{\rm th}}      
\newcommand{\Fbarth}{\bar F_{\rm th}}
\newcommand{\kth}{k_{\rm th}}      
\newcommand{\xith}{\xi_{\rm th}}   
\newcommand{\gth}{g_{\rm th}}      

\begin{document}

\title{Hard-core thinnings of germ--grain models with power-law grain sizes}


\author{
Mikko Kuronen\thanks{Email: \protect\mailto{mikko.p.o.kuronen@jyu.fi}} \and Lasse Leskelä\thanks{
URL: \url{http://www.iki.fi/lsl/} \quad
Email: \protect\mailto{lasse.leskela@iki.fi}}
}
\date{\today}
\maketitle

\begin{abstract}
Random sets with long-range dependence can be generated using a Boolean model with power-law grain
sizes. We study thinnings of such Boolean models which have the hard-core property that no grains
overlap in the resulting germ--grain model. A fundamental question is whether long-range
dependence is preserved under such thinnings. To answer this question we study four natural
thinnings of a Poisson germ--grain model where the grains are spheres with a regularly varying
size distribution. We show that a thinning which favors large grains preserves the slow
correlation decay of the original model, whereas a thinning which favors small grains does not.
Our most interesting finding concerns the case where only disjoint grains are retained, which
corresponds to the well-known Matérn type~I thinning. In the resulting germ--grain model, typical
grains have exponentially small sizes, but rather surprisingly, the long-range dependence property
is still present. As a byproduct, we obtain new mechanisms for generating homogeneous and
isotropic random point configurations having a power-law correlation decay.
\end{abstract}



\section{Introduction}

Consider a random closed set which can be expressed as a union of compact sets in the
$d$-dimensional Euclidean space $\R^d$. The compact building blocks of the random set are called
grains, the collection of grains \emph{germ--grain model}, and the union of grains \emph{grain
cover}. A germ--grain model is called \emph{hard-core} if the grains are disjoint with probability
one. Hard-core germ--grain models (a.k.a.\ random packing models) provide an important class of
mathematical tools for the natural sciences, allowing to model and analyze the statistical
features of disordered porous materials \cite{Ohser_Mucklich_2000,Schuth_Sing_Weitkamp_2002}.
Besides natural sciences, these models have found applications in engineering when analyzing the
performance of medium access protocols in wireless data networks (e.g.
\cite{Baccelli_Blaszczyszyn_2009b,Haenggi_2011,Nguyen_Baccelli}).

A key statistical feature of a random set is its covariance function, which describes how much
more or less likely it is to find matter at a given distance from a location containing matter,
compared to finding matter in an arbitrary location. While most germ--grain models studied in the
literature have a rapidly decaying covariance function, certain experimental studies in astronomy
\cite{Jones_Martinez_Saar_Trimble_2005} and materials science~\cite{Schuth_Sing_Weitkamp_2002}
display real-world data where the statistically estimated covariance function appears to decay
exceptionally slowly, following a power law $r^{-\beta}$ with some exponent $\beta > 0$ for large
distances $r$. When $\beta < d$, such models are \emph{long-range dependent} in the sense that
\begin{equation}
 \label{eq:LRD}
 \limsup_{r \to \infty} \frac{\var( |X \cap B_r| )}{ r^d } = \infty,
\end{equation}
where $|X \cap B_r|$ denotes the volume of the region covered by the random set $X$ within the
closed ball $B_r$ with radius $r$ centered at the origin \cite[Sec~12.7]{Daley_Vere-Jones_2008}.
Long-range dependence causes anomalous behavior to several statistical features of the model, as
is well understood in time series analysis \cite{Samorodnitsky_2006}. Note that for a homogeneous
random set in dimension $d=1$, property \eqref{eq:LRD} is equivalent to the usual notion of
long-range dependence,
\[
 \limsup_{n \to \infty} \frac{\var( \sum_{k=1}^n X_k)}{ n } = \infty,
\]
of the time series $X_k = | X \cap (k-1,k] |$.

Our goal in this article is to construct parsimonious germ--grain models having the hard-core and
long-range dependence property. In the presence of long-range dependence, the requirement of
parsimony, i.e.\ having a small number of model parameters, is especially important because
long-range dependence tends to reduce the robustness of the statistical estimators of model
parameters  \cite{Clauset_Shalizi_Newman_2009}. Long-range dependent germ--grain models are easy
to generate using a Boolean model---a germ--grain model with random power-law distributed sizes
and independently and uniformly scattered centers---but the resulting model is not hard-core by
construction. To make it hard-core, we shall follow Matérn's approach \cite{Matern_1960} of
thinning out a selected collection of overlapping grains from the proposed Boolean model so that
the resulting collection of grains is disjoint. Whether this approach is feasible for obtaining
hard-core models with long-range dependence depends on the following question:
\begin{quote}\it
  Is the power-law covariance decay of the proposed Boolean model
  preserved after making it disjoint by thinning?
\end{quote}

To answer the above question, we shall analyze in detail the following natural thinning
mechanisms:
\begin{itemize}
  \item \emph{Large retained}. Let the thinned model consist of those grains in the original Boolean model
  which are not overlapped by any larger grain in the original model.
  \item \emph{Random retained}. Assign independent random weights to the grains.
  Let the thinned model consist of those grains in the original model
  which are not overlapped by any heavier grain in the original model. (This thinning corresponds to Matérn type
  II.)
  \item \emph{Small retained}. Let the thinned model consist of those grains in the original model
  which are not overlapped by any smaller grain in the original model.
  \item \emph{Isolated retained}. Let the thinned model be the set of grains in the original model
  which do not overlap with any other grain in the original model.
  (This thinning corresponds to Matérn type I.)
\end{itemize}
We remark that---unlike the Matérn type III hard-core model \cite{Nguyen_Baccelli}---the above
thinnings are local in that the decision whether a proposed grain shall be retained or not is made
solely by looking at the grains which intersect it.

For simplicity, we shall restrict to spherical models where the grains are closed balls.
Figure~\ref{fig:ThinnedSets} illustrates the above four thinnings applied to a simulated sample of
a Boolean model in $\R^2$ where the grain centers have mean density $\lambda = 0.05$ and the grain
radii have a Pareto distribution $F(r) = 1-r^{-\alpha}$, $r \ge 1$, with tail exponent
$\alpha=2.5$.

\begin{figure}[h!]
 \includegraphics[width=\textwidth]{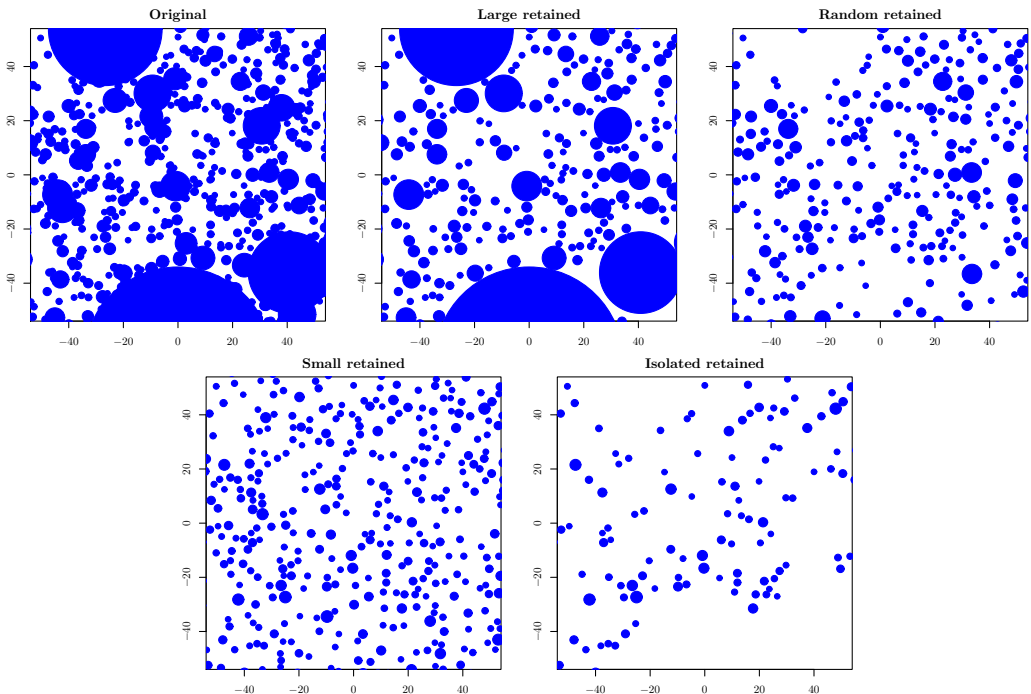}
 \caption{\label{fig:ThinnedSets} Original model and its four hard-core thinnings, where the original model is
 generated by disks having a Pareto distribution with tail exponent $\alpha = 2.5$.}
\end{figure}

The above thinnings will be analyzed collectively by viewing them as instances of a general
weight-based thinning mechanism, following Månsson and Rudemo \cite{Mansson_Rudemo_2002}. Standard
formulas of Palm calculus allow to write down closed-form analytical formulas for the radius
distribution of a typical grain, the covariance function of the grain cover, and the two-point
correlation function of the grain centers for general hard-core germ--grain configurations
generated by weight-based thinnings. Using the theory of regular variation, we analyze the
long-range behavior of these quantities under the assumption that the grain radii in the proposed
Boolean model follow a power-law distribution with tail exponent $\alpha > d$.

The main results of this article (Theorems \ref{the:isolated}, \ref{the:random}, \ref{the:large},
\ref{the:small}) are summarized in Table~\ref{tab:Results} below.
\begin{table}[h]
\small \centering
\begin{tabular}{l|>{}m{9em}|>{}m{9em}|>{}m{9em}}
Model & Radius distribution of a typical grain & Covariance function of grain cover & Correlation
function of grain centers
\tabularnewline
\hline \hline
Original   & power law $(\alpha)$ & power law $(\alpha - d)$ & zero
\tabularnewline
Large retained & power law $(\alpha)$ & power law $(\alpha - d)$ & power law $(\alpha - d)$
\tabularnewline
Random retained & power law $(\alpha + d)$ & power law $(\alpha - d)$ & power law $(\alpha - d)$
\tabularnewline
Small retained & exponential  & exponential & exponential
\tabularnewline
Isolated retained & exponential & power law $(\alpha - d)$ & power law $(\alpha - d)$
\tabularnewline
\end{tabular}
\caption{\label{tab:Results} Long-range decay of key statistical characteristics of the original Boolean model and the
hard-core germ--grain models obtained by thinning.}
\end{table}
From the table, we can draw the following conclusions:
\begin{itemize}
  \item The power-law covariance decay and long-range dependence (when $\alpha < 2d$)
  of the grain cover are preserved under all
  thinnings except \emph{small retained}.
  \item Whereas the random point configuration (a.k.a.\ point process) of grain centers
  in the proposed Boolean model is completely uncorrelated,
  the corresponding point configurations in all thinned models except \emph{small retained}
  have a power-law two-point correlation function.
  \item The heavy tail of the grain radius distribution is destroyed by \emph{small
  retained} and \emph{isolated retained} thinnings. The other two thinnings preserve the
  power-law structure of the tail distribution: under \emph{large retained} with the same
  exponent, under \emph{random retained} with a larger exponent corresponding to a lighter tail.
\end{itemize}
Table~\ref{tab:Results} also reveals a striking feature of the \emph{isolated retained} thinning
mechanism: The resulting grain cover and the resulting point configuration of grain centers both
exhibit long-range dependence although the grain size distribution is light-tailed. This seemingly
paradoxical phenomenon can be explained by inspecting the empty space: Any region of space not
covered by the thinned germ--grain model is likely to have been contained in a big grain of the
proposed model that was removed in the thinning, and therefore, a large neighborhood of this empty
region is likely to be empty, too.

This article may be seen as a continuation of the works of Månsson and Rudemo
\cite{Mansson_Rudemo_2002} and Andersson, Häggström and Månsson
\cite{Andersson_Haggstrom_Mansson_2006}, who analyzed first-order statistical properties of
hard-core germ--grain models obtained by weight-based thinnings. In
\cite[Cor~3.1]{Mansson_Rudemo_2002} it was also shown that \emph{large retained} thinning
preserves the tail behavior of the typical grain radius whenever the proposed grain radius
distribution is continuous. A slightly more general thinning framework was recently introduced by
Nguyen and Baccelli \cite{Nguyen_Baccelli}, who derived differential equations characterizing the
generating functional of the random point configuration formed by the thinned grain centers.
Earlier work on the covariance analysis of random sets includes Böhm and
Schmidt~\cite{Bohm_Schmidt_2003}, who derived a short-range approximation for the covariance
function of a general homogeneous random set. Snethlage, Martínez, Stoyan,
Saar~\cite{Snethlage_Martinez_Stoyan_Saar_2002} (see also references therein) provide a nice
summary of random point configuration models where the two-point correlation function has a
power-law behavior on short distances. Earlier works on long-range dependent random sets appear
mostly restricted to random point configuration in dimension $d=1$. Among these, Daley and Vesilo
\cite{Daley_Vesilo_1997} established the following elegant preservation property for many queueing
systems: the point configuration of the departure times is long-range dependent if and only if the
same is true for the arrival times. Daley~\cite{Daley_1999} showed that a renewal point process is
long-range dependent if the interpoint distances have an infinite second moment, and Kulik and
Szekli~\cite{Kulik_Szekli_2001} extended this observation to one-dimensional point configurations
with positively associated interpoint distances. Vamvakos and Anantharam
\cite{Vamvakos_Anantharam_1998} showed that the long-range dependence of a point process is
preserved by a leaky bucket flow control mechanism for data traffic. A study focused on the
long-range dependence of multidimensional random sets is the recent work of Demichel, Estrade,
Kratz, and Samorodnitsky \cite{Demichel_Estrade_Kratz_Samorodnitsky_2011}, who studied whether
random sets having power-law decaying chord length distributions, closely related to the
covariance function of the random set, can be generated as a level set of a Gaussian random
field---they found that in wide generality (merely assuming that the underlying Gaussian field is
mixing), this is not possible.

Let us summarize the notational conventions used in this paper. The symbol $\pr$ stands for the
probability measure on some abstract probability space which governs all randomness in the models,
and $\E, \var, \cov$ denote the expectation, variance, and covariance with respect to $\pr$,
respectively. The symbol $B_r(x)$ denotes the closed unit ball with center $x$ and radius $r$ in
the $d$-dimensional Euclidean space $\R^d$. We use $B_r$ as shorthand for $B_r(o)$, where $o$ is
the origin of $\R^d$. For a Borel set $B$ in $\R^d$, we denote by $|B|$ its Lebesgue measure, and
by $\ind_B(x)$ or $\ind(x \in B)$ its indicator function. The symbols $dx$, $dy$, etc.\ refer to
the Lebesgue measure in $\R^d$. The symbol $\R_+$ denotes the positive real numbers including
zero. The symbol $F(dr)$ refers to integration with respect to a probability measure $F$ on
$\R_+$, whereas $F(r) = F[0,r]$ and $\bar F(r) = 1-F(r)$ stand for the corresponding cumulative
distribution function and the complementary cumulative distribution function, respectively. The
minimum and maximum of real numbers $a$ and $b$ are denoted by $a \wedge b$ and $a \vee b$,
respectively. When convenient, we denote $\int_a^\infty = \int_{(a,\infty)}$, $\int_0^b =
\int_{[0,b]}$, and $\int_a^b = \int_{(a,b]}$ for $0<a<b<\infty$. For functions $f$ and $g$ defined
on the positive real line, we denote $f \sim g$ if $f(t)/g(t) \to 1$ as $t \to \infty$.

The rest of the paper is organized as follows. Section~\ref{sec:BooleanModels} summarizes
preliminaries on random Boolean models needed later in the text. Section~\ref{sec:Thinning}
introduces a weight-based thinning mechanism which produces hard-core germ--grain models from
Boolean models and list formulas for the second-order statistics of the models so obtained.
Section~\ref{sec:AsymptoticSecondOrder} contains a long-range analysis of the second-order
statistics of the previous section. The main results of Table~\ref{tab:Results} are proved
case-by-case in Section~\ref{sec:isolated} (isolated retained), Section~\ref{sec:random} (random
retained), Section~\ref{sec:large} (large retained), and Section~\ref{sec:small} (small retained.
Section~\ref{sec:Conclusions} concludes the paper.

\section{Boolean models with power-law grain radii}
\label{sec:BooleanModels}

A spherical Boolean model is a random collection of closed spheres, where the sphere centers are
independently and uniformly scattered in $\R^d$ and the sphere radii are independent and
identically distributed random variables in $\R_+$. Mathematically, a spherical Boolean model can
be defined as a Poisson random measure $\Phi$ on $\R^d \times \R_+$ with intensity measure
$\lambda dx F(dr)$, where $\lambda$ is a positive constant and $F$ is a probability measure on
$\R_+$ such that $\int r^d F(dr) < \infty$. We identify each pair $(x,r) \in \Phi$ with the closed
ball $B_r(x)$ with center $x$ and radius $r$ and---conforming to the terminology of more general
germ--grain models---such pairs wills be called \emph{grains}. The random closed set
\[
 X = \!\! \bigcup_{(x,r) \in \Phi} B_r(x)
\]
is called the \emph{grain cover} of $\Phi$, and we denote by
\[
 \Phi_g = \{x \in \R^d: (x,r) \in \Phi \ \text{for some $r$} \}
\]
the random point configuration in $\R^d$ formed by the grain centers of $\Phi$. Note that $\Phi_g$
is a homogeneous Poisson random measure on $\R^d$ with intensity measure $\lambda dx$. The
parameter $\lambda$ thus equals the mean density of grain centers, and the probability measure $F$
is the common distribution of grain radii. For general definitions and details about random sets
and random measures, see for example
\cite{Daley_Vere-Jones_2008,Molchanov_2005,Schneider_Weil_2008,Stoyan_Kendall_Mecke_1995}.

The covariances of the random set $X$ are denoted by $k(x,y) = \cov( \ind_X(x), \ind_X(y) )$,
where $\ind_X$ is the indicator function of $X$. Because the distribution of $X$ is
shift-invariant by construction, the covariances are given by $k(x,y) = k(x-y)$, where the
covariance function $k(z) = k(o,z)$ is given by the well-known formula (e.g.
\cite[Sec.~3.1]{Stoyan_Kendall_Mecke_1995})
\begin{equation}
 \label{eq:BooleanCov}
 k(z) = (1-p)^2 \left( e^{\lambda \int | B_r(o) \cap B_r(z) | \, F(dr) } - 1 \right),
\end{equation}
and where $p$ is the volume fraction of $X$ given by
\[
 p = 1 - e^{-\lambda \int |B_r| \, F(dr)}.
\]
Formula~\eqref{eq:BooleanCov} indeed shows that $k(z)$ depends on $z$ only through $|z|$, which is
evident because $X$ is isotropic by construction. Using this formula we may also deduce that
\[
 k(z) \, \sim \, \lambda (1-p)^2 \int | B_r(o) \cap B_r(z) | \, F(dr)
 \quad \text{as $|z| \to \infty$},
\]
where we denote $f(z) \sim g(z)$ if $f(z)/g(z) \to 1$ as $|z| \to \infty$. When the grain radius
distribution $F$ follows a power law with tail exponent $\alpha > d$, so that $F(r) = 1 - \ell(r)
r^{-\alpha}$ for some slowly varying function $\ell$ (see Appendix~\ref{sec:RegularVariation} for
details), it follows by using Lemma~\ref{lemma_average_intersection} in
Appendix~\ref{sec:Intersections} that
\[
 k(z) \, \sim \, \lambda  (1-p)^2 c_{\alpha,d} \ell(|z|) |z|^{-(\alpha-d)}
 \quad \text{as $|z| \to \infty$}.
\]
Thus, when the radius distribution follows a power law with tail exponent $\alpha > d$, then the
covariance function $k(z)$ follows a power law with tail exponent $\alpha - d$. Especially, the
Boolean grain cover $X$ is long-range dependent in the sense of~\eqref{eq:LRD} for $\alpha \in
(d,2d)$.

\section{Weight-based thinning}
\label{sec:Thinning}

In this section we shall study a weight-based thinning mechanism which maps a Boolean model into a
hard-core germ--grain model consisting of nonoverlapping grains
\cite{Mansson_Rudemo_2002,Nguyen_Baccelli}. This thinning mechanism is defined by assigning random
weights to the grains of the Boolean model, and retaining those grains which are not overlapped by
any other grain in the Boolean model with a higher or equal weight.

\subsection{Thinning mechanism}

A weighted spherical Boolean model is defined as a Poisson random measure $\Phi$ on $\R^d \times
\Rp \times \Rp$ with intensity measure
\[
 \Lambda(dx,dr,dw) = \lambda dx F(dr) G_r(dw),
\]
where $\lambda > 0$, $F$ is a probability measure on $\R_+$ such that $\int r^d F(dr) < \infty$,
and $G$ is a probability kernel on $\R_+$ (a family of probability measures $G_r$ on $\R_+$
indexed by $r$ such that $r \mapsto G_r(A)$ is measurable for measurable $A \subset \R_+$). A
triplet $(x,r,w) \in \Phi$ is identified as a grain with center $x$, radius $r$, and weight $w$.
As in Section~\ref{sec:BooleanModels}, the constant $\lambda$ is the mean density of grain centers
and the probability measure $F$ is the distribution of grain radii. The probability measure $G_r$
is the weight distribution of a grain with radius $r$.

We say that two distinct grains are \emph{neighbors} if they intersect each other, and we denote
the set of neighbors of a reference grain $(x,r,w)$ by
\begin{equation}
 \label{eq:Neighbors}
 N_{x,r,w}
 = \left\{ (x',r',w') \in \R^d \times \Rp \times \Rp \setminus \{(x,r,w)\} :
   B_{r'}(x') \cap B_r(x) \neq \emptyset \right\}.
\end{equation}
The thinning of a weighted spherical Boolean model $\Phi$ is now defined by $\Phith = T(\Phi)$,
where
\begin{equation}
 \label{eq:Thinning}
 T(\Phi)
 = \left\{(x, r, w) \in \Phi : w > w' \ \text{for all} \ (x',r',w') \in \Phi \cap N_{x,r,w} \right\}.
\end{equation}
To rephrase the definition, we say that a grain $(x',r',w')$ \emph{obstructs} grain $(x,r,w)$ if
$(x',r',w')$ is a neighbor of $(x,r,w)$ and $w' \ge w$. Then by definition, the thinned
germ--grain configuration $\Phith$ consists of grains in $\Phi$ which are not obstructed by any
other grain in $\Phi$. Note that two overlapping grains with equal weights obstruct each other,
and will be both removed.

The following choices of $G_r$ yield the four thinnings which shall be analyzed in detail in
Sections~\ref{sec:isolated}--\ref{sec:small}.
\begin{itemize}
  \item \emph{Large retained}. The weight of each grain is set equal to its radius, so that $G_r(dw) = \delta_r(dw)$.
  \item \emph{Random retained}. The grains are assigned independent uniformly distributed random weights, so that
  $G_r(dw) = 1_{(0,1)}(w) dw$.
  \item \emph{Small retained}. The weight of each grain is set equal to the inverse of its radius, so that $G_r(dw) = \delta_{1/r}(dw)$.
  \item \emph{Isolated retained}. All grains are assigned weight one, so that $G_r(dw) = \delta_1(dw)$.
\end{itemize}

\subsection{Retention probability}
The \emph{retention probability} of a reference grain $(x,r,w)$ is defined as the probability that
$(x,r,w)$ belongs to the germ--grain configuration obtained by thinning the union $\Phi \cup
\{(x,r,w)\}$. Because this probability does not depend on $x$ (see
Proposition~\ref{the:RetentionProbability}), we shall denote it by
\[
 h(r,w) = \pr( (x,r,w) \in T( \Phi \cup \{(x,r,w)\} ) ).
\]
The quantity $h(r,w)$ may be regarded as the probability that a typical grain with radius $r$ and
weight $w$ in the proposed Boolean model is retained (see e.g.~\cite{Schneider_Weil_2008,Stoyan_Kendall_Mecke_1995}). Analogously, the
weight-averaged retention probability
\begin{equation}
 \label{eq:H1Integrated}
 h(r) = \int_\Rp h(r, w) \, G_r(dw)
\end{equation}
may be regarded as the probability that a typical grain of radius $r$ in the proposed Boolean
model is retained. The following result \cite[Thm.~2.2]{Mansson_Rudemo_2002} gives a formula for
the retention probability. For the reader's convenience we will include the proof here.

\begin{proposition}
\label{the:RetentionProbability}
The retention probability of an arbitrary reference grain $(x,r,w) \in \R^d \times \Rp \times \Rp$
does not depend on $x$, and is given by
\begin{equation}
 \label{eq:H1}
 h(r, w) = \exp \left\{ - \lambda \int_\Rp | B_{r + s}(o) | \, G_s[w, \infty) \, F(ds) \right\}.
\end{equation}
\end{proposition}
\begin{proof}
Fix a reference grain $(x,r,w)$ and denote $\Phi' = \Phi \cup \{(x,r,w)\}$. By definition, the
reference grain belongs to the thinned configuration $T(\Phi')$ if and only if $w > w'$ for all
$(x',r',w') \in N_{x,r,w} \cap \Phi'$, where $N_{x,r,w}$ is the neighbor set of $(x,r,w)$ defined
by~\eqref{eq:Neighbors}. Observe that $N_{x,r,w} \cap \Phi' = N_{x,r,w} \cap \Phi$, because no
grain is its own neighbor by definition. As a consequence, the retention probability can be
expressed using the the intensity measure of the Poisson point configuration $\Phi$ according to
\[
 \pr ( (x,r,w) \in T(\Phi') )
 = \pr( \Phi(A_{x,r,w}) = 0 )
 = e^{ - \Lambda(A_{x,r,w}) },
\]
where
\[
 A_{x,r,w} = \{ (x',r',w') \in N_{x,r,w} : w' \ge w \}
\]
is the set of grains obstructing $(x,r,w)$. The claim now follows because
\begin{align*}
 \Lambda(A_{x,r,w})
 &= \int_{\R^d} \int_{\Rp} \int_{\Rp} \ind(|x-x'| \le r + r') \, \ind(w' \ge w) \, G_{r'}(dw') F(dr') \lambda dx' \\
 &= \lambda \int_{\Rp} |B_{r+r'}(o)| \, G_{r'}[w,\infty) F(dr').
\end{align*}
\end{proof}

\subsection{First-order statistics of the thinned model}

Let us summarize some key formulas about the first-order statistics of the thinned germ--grain
model $\Phith$ which were obtained in \cite{Andersson_Haggstrom_Mansson_2006,Mansson_Rudemo_2002}.
The mean density of grain centers in the thinned model is given by
\begin{equation}
 \label{eq:GermDensity}
 \lambdath = \lambda \int_{\Rp} h(r) \, F(dr),
\end{equation}
where $h(r)$ is the weight-averaged retention probability defined in~\eqref{eq:H1Integrated}, and
the radius distribution of a typical grain in the thinned model equals
\begin{equation}
 \label{eq:radiusth}
 \Fth (r) = 1 - \frac{\lambda}{\lambdath} \int_r^\infty h(s) \, F(ds).
\end{equation}
Moreover, the volume fraction of the thinned grain cover
\[
 \Xth = \!\! \bigcup_{(x,r,w) \in \Phith} B_r(x)
\]
is given by
\begin{equation}
 \label{eq:VolumeFraction}
 \pth = \lambda \int_\Rp |B_r| h(r) F(dr).
\end{equation}
Note that the quantity $\int h(r) F(dr)$ in~\eqref{eq:GermDensity} may be regarded as the
probability that a randomly chosen grain in the proposed Boolean model is retained by the thinning
mechanism.

\subsection{Pair retention probability}

The \emph{pair retention probability} of a given pair of reference grains $(x_1,r_1,w_1)$ and
$(x_2,r_2,w_2)$ is defined as the probability that both reference grains belong to the germ--grain
configuration obtained by thinning the union $\Phi' = \Phi \cup \{(x_1,r_1,w_1), (x_2,r_2,w_2)\}$.
Because this probability depends on $x_1$ and $x_2$ only through their distance (see
Proposition~\ref{the:H2}), we shall denote it by
\begin{equation}
 \label{eq:H2}
 h_2(u,r_1,w_1,r_2,w_2) = \pr( \{(x_1,r_1,w_1),(x_2,r_2,w_2)\} \in T( \Phi' ) ),
\end{equation}
where $u=|x_1-x_2|$. The weight-averaged pair retention probability is defined by
\begin{equation}
 \label{eq:H2Int}
 h_2(u,r_1,r_2) = \int_\Rp \int_\Rp h_2(u,r_1,w_1,r_2,w_2) \, G_{r_1}(dw_1) G_{r_2}(dw_2).
\end{equation}

\begin{proposition}
\label{the:H2}
The pair retention probability of two reference grains $(x_1, r_1, w_1)$ and $(x_2, r_2, w_2)$
depends on $x_1$ and $x_2$ only through the distance $u=|x_1-x_2|$. For $u \le r_1 + r_2$ this
probability equals zero, and for $u > r_1 + r_2$,
\[
 h_2(u, r_1, w_1, r_2,w_2)
 = h(r_1, w_1) h(r_2, w_2) e^{\tau(u, r_1, w_1, r_2, w_2)}
\]
where $h(r_1, w_1)$ and $h(r_2, w_2)$ are the retention probabilities defined by~\eqref{eq:H1},
and
\[
 \tau(u, r_1, w_1, r_2, w_2) =
 \lambda \int_\Rp | B_{r + r_1}(x_1) \cap B_{r + r_2}(x_2) | \, G_r[w_1 \vee w_2, \infty ) \, F(dr)
\]
is the mean number of grains in $\Phi$ which simultaneously obstruct both reference grains.
\end{proposition}
\begin{proof}
Fix two reference grains $(x_1, r_1, w_1)$ and $(x_2, r_2, w_2)$ and assume that they do not
overlap, so that $|x_1-x_2| > r_1 + r_2$. Denote $\Phi' = \Phi \cup \{(x_1, r_1, w_1), (x_2, r_2,
w_2)\}$. Recall that grain $(x_1,r_1,w_1)$ belongs to $T(\Phi')$ if and only if $w_1
> w$ for all $(x,r,w) \in N(x_1,r_1,w_1) \cap \Phi'$. Because no grain is its own neighbor by
definition, and because the two reference grains are not neighbors, we see that $N(x_1,r_1,w_1)
\cap \Phi' = N(x_1,r_1,w_1) \cap \Phi$. By symmetry, a similar conclusion also holds for the other
reference grain.

We conclude that for $i=1,2$, grain $(x_i,r_i,w_i)$ is retained if and only if $\Phi(A_i) =
\emptyset$, where
\[
 A_i = \{ (x,r,w) \in N(x_i,r_i,w_i): w \ge w_i \}
\]
is the set of grains obstructing $(x_i,r_i,w_i)$. Now the pair retention probability can be
written as
\begin{equation}
 \label{eq:h2Prob}
 h_2 = \pr ( \Phi(A_1 \cup A_2) = 0 ).
\end{equation}
The number of grains in $\Phi \cap (A_1 \cup A_2)$ is Poisson distributed with mean
\[
 \Lambda(A_1 \cup A_2) = \Lambda(A_1) + \Lambda(A_2) - \Lambda(A_1 \cap A_2).
\]
Because $e^{-\Lambda(A_i)}$ equals the retention probability $h(r_i,w_i)$ of grain $(x_i,r_i,w_i)$
(see Proposition~\ref{the:RetentionProbability}), we see that
\[
 h_2 = h(r_1,w_1) h(r_2,w_2) e^{\Lambda(A_1 \cap A_2)}.
\]
The claim now follows after noting that
\[
 \Lambda(A_1 \cap A_2)
 = \lambda \int_\Rp | B_{r_1+r}(x_1) \cap B_{r_2+r}(x_2) | G_r[w_1 \vee w_2, \infty ) \, F(dr).
\]
\end{proof}

A key quantity for analyzing the covariance function of the thinned grain cover in
Section~\ref{sec:AsymptoticSecondOrder} is the following function, which we shall call the
\emph{retention covariance function}. It is defined by
\begin{equation}
 \label{eq:RetentionCovariance}
 q(u,r_1,r_2) = h_2(u,r_1,r_2) - h(r_1) h(r_2),
\end{equation}
where $h(r)$ denotes the weight-averaged retention probability defined in~\eqref{eq:H1Integrated},
and $h_2(u, r_1, r_2)$ is the weight-averaged pair retention probability defined
in~\eqref{eq:H2Int}.

\begin{lemma}
\label{the:H2Bound}
The retention covariance function satisfies
\[
 \left| q(u,r_1,r_2 ) \right| \le h(r_1) \wedge h(r_2)
\]
for all $u, r_1, r_2 \ge 0$.
\end{lemma}
\begin{proof}
Fix a pair of reference grains $(x_1,r_1,w_1)$ and $(x_2,r_2,w_2)$ having their centers at a
distance $u=|x_1-x_2|$ apart. Define a weight-dependent version of $q$ by
\[
 q(u,r_1,w_1,r_2,w_2) = h_2(u,r_1,w_1,r_2,w_2) - h(r_1,w_1) h(r_2,w_2).
\]
We will first show that
\begin{equation}
 \label{eq:H2Bound1}
 |q(u,r_1,w_1,r_2,w_2)| \le h(r_1,w_1),
\end{equation}
by separately considering the following two cases:
\begin{enumerate}[(i)]
  \item If $u \le r_1 + r_2$, then $h_2(u,r_1,w_1,r_2,w_2 )$ is zero because the reference grains
overlap, and~\eqref{eq:H2Bound1} follows immediately.
  \item If $u > r_1 + r_2$, then by borrowing the notation from the proof of Proposition~\ref{the:H2}, we
have by~\eqref{eq:h2Prob} that
\[
 0 \le h_2(u,r_1,w_1,r_2,w_2)
 = \pr( \Phi(A_1 \cup A_2) = 0 )
 \le \pr( \Phi(A_1) = 0 )
 = h(r_1,w_1).
\]
As a consequence,
\[
 -h(r_1,w_1) h(r_2,w_2) \le q(u,r_1,w_1,r_2,w_2) \le h(r_1,w_1) (1-h(r_2,w_2)),
\]
from which \eqref{eq:H2Bound1} again follows.
\end{enumerate}
After integrating both sides of \eqref{eq:H2Bound1} over the weights, we see that $| q(u,r_1,r_2 )
| \le h(r_1)$. By symmetry, the same inequality holds with $r_1$ replaced by $r_2$, which proves
the claim.
\end{proof}

\subsection{Covariance function of the thinned grain cover}

Let us now consider the covariance function
\[
 \kth(z) = \pr( o \in \Xth \!\!, \ z \in \Xth ) - \pr(o \in \Xth) \pr(z \in \Xth)
\]
of the thinned grain cover $\Xth$.

\begin{proposition}
\label{the:CovThinnedGrainCover}
The covariance function of the thinned grain cover is given by
\begin{equation}
  \label{eq:cov}
  \begin{aligned}
   &\kth(z) \ = \ \lambda \int_\Rp |B_r(o) \cap B_r(z)| h(r)  F(dr) \\
   & + \lambda^2 \int_\Rp \int_\Rp \int_{\R^d} |B_{r_1}(o) \cap B_{r_2}(x)| \, q(|x - z|, r_1, r_2)
   \, dx F(dr_1) F(dr_2),
  \end{aligned}
\end{equation}
where $h$ is the weight-averaged retention probability defined by~\eqref{eq:H1Integrated} and $q$
is the retention covariance function defined by~\eqref{eq:RetentionCovariance}.
\end{proposition}
\begin{proof}
Let us express the covariance function as
\[
 \kth(z) = S_1(z) + S_2(z) - \pth^2,
\]
where $S_1(z)$ is the probability that a single grain in $\Phith$ simultaneously covers $o$ and
$z$, $S_2(z)$ is the probability that $o$ and $z$ are covered by distinct grains in $\Phith$, and
the volume fraction $\pth$ can be viewed as the probability that an arbitrary reference point in
$\R^d$ is covered by some grain in $\Phith$.

To write down an analytical expression for $S_1(z)$, recall first that by the hard-core property,
the indicator function of $\Xth$ can be written as
\[
 \ind_{\Xth}(y) = \sum_{(x,r,w) \in \Phi} f_y(\Phi; x,r,w),
\]
where $f_y(\Phi; x,r,w) = \ind(y \in B_r(x)) \ind_{T(\Phi)}(x,r,w)$ is the indicator for the event
that a grain $(x,r,w)$ covers $y$ and is contained in $\Phith$. Then
\[
 S_1(z) = \E \sum_{\mathclap{(x,r,w) \in \Phi}} f_o(\Phi; x,r,w) f_z(\Phi; x,r,w),
\]
Using Mecke's formula \cite[Thm. 3.2.5]{Schneider_Weil_2008} it's easy to see that
\[
 S_1(z) = \lambda \int_\Rp |B_r(o) \cap B_r(z)| \, h(r) \, F(dr),
\]
where $h(r) = \int h(r,w) \, G_r(dw)$.

The probability that $o$ and $z$ are covered by distinct grains in $\Phith$ can analogously be
written as
\[
 S_2(z)
 = \E \sum_{\substack{(x_1, r_1, w_1) \in \Phi \\ (x_2, r_2, w_2) \in \Phi \\ (x_1,r_1,w_1) \neq (x_2,r_2,w_2)}}
 f_o( \Phi; x_1,r_1,w_1 ) f_z ( \Phi; x_2,r_2,w_2 ).
\]
Using the Slivnyak--Mecke formula \cite[Cor.~3.2.3]{Schneider_Weil_2008}, it's not hard to verify that
\[
 S_2(z)
 = \lambda^2 \iiint |B_{r_1}(o) \cap B_{r_2}(x)| \, h_2(|x - z|, r_1, {r_2}) \, dx F(dr_1) F(dr_2),
\]
where $h_2$ is the pair retention probability defined by~\eqref{eq:H2}. The validity of the claim
now follows after representing $\pth$ using \eqref{eq:VolumeFraction} and the identity
$|B_{r_1}(o)| |B_{r_2}(o)| = \int |B_{r_1}(o) \cap B_{r_2}(x)| \, dx$ to note that
\begin{align*}
 \pth^2
 &= \lambda^2 \iiint |B_{r_1}(o) \cap B_{r_2}(x)| \, h(r_1) h(r_2) \, dx F(dr_1) F(dr_2).
\end{align*}
\end{proof}

\subsection{Two-point correlation function of thinned grain centers}

The \emph{two-point correlation function} $\xith(z)$ of the random point configuration $\Phith_g =
\{x: (x,r,w) \in \Phith\}$ of the thinned grain centers is defined as a function which satisfies
\[
 \cov ( \Phith_g ( A ), \Phith_g ( B ) ) = \lambdath^2 \int_A \int_B \xith(x - y) \, dx dy
\]
for all disjoint and bounded measurable sets $A, B \subset \R^d$, assuming such function exists.
This function, which in our case only depends on $|z|$, describes how much more ($\xith(z) > 0$)
or less ($\xith(z) < 0$) likely it is to observe a point at a distance $|z|$ from a typical point,
compared to observing a point in an arbitrary location. The two-point correlation function is
related to the pair-correlation function $\gth$ commonly used in statistics, via the formula
$\xith(z) = \gth(z) - 1$ (e.g.\ \cite{Stoyan_Kendall_Mecke_1995}).

\begin{proposition}
\label{the:TwoPointCorrelation}
The two-point correlation function of the thinned grain centers is given by
\begin{equation}
 \label{eq:TwoPointCorrelation}
 \xith(z) = \frac{\lambda^2}{\lambdath^2} \int_\Rp \int_\Rp q(|z|,r_1,r_2) \, F(dr_1) F(dr_2),
\end{equation}
where $\lambdath$ is the thinned germ density defined in~\eqref{eq:GermDensity}, and $q$ is the
retention covariance function defined in~\eqref{eq:RetentionCovariance}.
\end{proposition}
\begin{proof}
By using the Slivnyak--Mecke formula~\cite[Cor.~3.2.3]{Schneider_Weil_2008} one can check that
\[
 \E \Phith_g ( A ) \Phith_g ( B )
 = \lambda^2 \int_A \int_B \int_{\R+} \int_{\R+} h_2(|x-y|, r_1,r_2) \, F(dr_1) F(dr_2) \, dx dy.
\]
for all bounded and disjoint $A,B \subset \R^d$, where $h_2$ is the weight-averaged pair retention
probability defined in~\eqref{eq:H2Int}. On the other hand, Mecke's formula~\cite[Thm.
3.2.5]{Schneider_Weil_2008} implies that
\[
 \E \Phith_g (A)
 = \lambda \int_A \int_{\R_+} h(r) \, F(dr) \, dx,
\]
where $h(r)$ is the weight-averaged retention probability defined in~\eqref{eq:H1Integrated}. The
claim follows by combining the above two formulas and recalling the definition of the retention
covariance function~\eqref{eq:RetentionCovariance}.
\end{proof}

\section{Long-range behavior of second-order statistics}
\label{sec:AsymptoticSecondOrder}

In this section we assume that the grain radius distribution $F$ of the proposed Boolean model
follows a power law with tail exponent $\alpha > d$, by which we mean that the complementary
cumulative distribution function $\bar F(r) = 1-F(r)$ is regularly varying at infinity with
exponent $-\alpha$. In this case we can write
\[
 \bar F(r) = \ell(r) r^{-\alpha},
\]
where the function $\ell$ is slowly varying at infinity (see Appendix~\ref{sec:RegularVariation}
for details).

\subsection{Asymptotic covariance}
The following result describes the covariance function of the thinned grain cover for thinnings
where large grains have small retention probability.

\begin{proposition}
\label{thm1}
Assume that the radius distribution $F$ follows a power law with tail exponent $\alpha > d$. Assume
that the weight-averaged retention probability $h(r)$ decays to zero as $r \to \infty$, and that
for any $r_1,r_2 \ge 0$, the retention covariance function defined
in~\eqref{eq:RetentionCovariance} decays according to
\begin{equation}
 \label{eq:CovDecoupling}
 q(|z|, r_1, r_2) \sim q_\infty(r_1,r_2) \bar F(|z|) |z|^d
 \quad \text{as $|z| \to \infty$}.
\end{equation}
Then the covariance function of the thinned grain cover decays according to
\[
 \kth(z) \sim c \bar F(|z|) |z|^d
 \quad \text{as $|z| \to \infty$},
 \]
where
\[
 c = \lambda^2 |B_1|^2 \int_\Rp \int_\Rp r_1^d r_2^d \, q_\infty(r_1,r_2) \, F(dr_1) F(dr_2).
\]
\end{proposition}

To prove Proposition~\ref{thm1} we need detailed results about the retention probabilities. The
following lemma allows us to use dominated convergence on a part of the domain.

\begin{lemma}
\label{lemma_main_part}
Assume that the radius distribution $F$ follows a power law with tail exponent $\alpha > d$. Then
there exist constants $c>0$ and $m > 0$ such that
\[
 0 \le q(|x - z|, r_1, r_2) \le c |z|^d \bar F (|z|)
\]
for all $x,z \in \R^d$ and all $r_1,r_2 \ge 0$ such that $|x| < r_1 + r_2 $, $|x - z|
\ge 2(r_1 + r_2)$, and $|z| > m$.
\end{lemma}
\begin{proof}
Let $c_1$ and $u_1$ be the constants from Lemma~\ref{lemma_int}.
Using the assumption that the function $\bar F$ follows a power law with tail exponent $\alpha$,
choose $u_2$ such that $\bar F(2/3 r) / \bar F(r) \le 2 (2/3)^{-\alpha}$ for all $r > u_1$.
Choose $u_3$ such that $\lambda c_1 r^d \bar F(r) \le 1$ for all $r > u_2$.
Note that $|x| \le r_1+r_2 \le \frac12 |x-z|$ implies $|z| \le |x| + |x-z| \le \frac{3}{2}
|x-z|$ and let $m = \max \{u_1, \frac32u_2, \frac32u_3\}$.
Using Proposition~\ref{the:H2} and the definition of $m$ we have for all $|z| > m$.
\begin{align*}
 q(|x - z|, r_1, r_2)
 &\le \exp \left(  \lambda \int_\Rp | B_{r_1+r}(o) \cap B_{r_2+r}(|x - z|) |  F(dr) \right) - 1\\
 &\le \exp \left(  \lambda c_1 |x-z|^d \bar F (|x-z|) \right) - 1\\
 &\le 2 \lambda c_1 |x-z|^d \bar F (|x-z|).
\end{align*}
Note that $|x - z| \le |x| + |z| \le \frac12 |x-z| + |z|$
implies $|x-z| \le 2|z|$, and that $\bar F$ is a decreasing function.
Now for $c = 4 (2/3)^{-\alpha} 2^d \lambda c_1$ and $|z| > m$ we have
\begin{align*}
 q(|x - z|, r_1, r_2)
 \le 2 \lambda c_1 (2|z|)^d \bar F (\tfrac23|z|)
 \le c |z|^d \bar F (|z|).
\end{align*}
\end{proof}

\begin{lemma}
\label{lemma_tail_part}
Fix $z \in \R^d$ and define
\[
 A(z) = \{ (x,r_1,r_2) \in \R^d \times \R_+ \times \R_+: |x-z| \le 2(r_1+r_2) \}.
\]
Then the retention covariance function $q$ satisfies
\begin{multline*}
 \iiint_{A(z)} |B_{r_1}(o) \cap B_{r_2}(x)| \, |q(|x-z|,r_1,r_2)| \, dx F(dr_1) F(dr_2) \\
 \le 2 |B_1|^2 \left( \int_{\R+} r^d F(dr) \right)
 \left( \int_{|z|/6}^\infty r^d F(dr) \right) \sup_{r>|z|/6} h(r),
\end{multline*}
where $h(r)$ is the weight-averaged retention probability defined by~\eqref{eq:H1Integrated}.
\end{lemma}
\begin{proof}
Define $d\mu$ as shorthand for $dx F(dr_1) F(dr_2)$, and denote the integrand by $f_z(x,r_1,r_2)$.
Observe that $f_z$ vanishes outside the set $A_0 = \{(x,r_1,r_2): |x| < r_1 + r_2\}$. Observe also
that $A(z) \cap A_0 \subset A_1(z) \cup A_2(z)$, where $A_i(z) = \{(x,r_1,r_2): r_i > |z|/6\}$. As
a consequence,
\[
 \int_{A(z)} f_z d\mu
 = \int_{A(z) \cap A_0} f_z d\mu
 \le \int_{A_1(z)} f_z d\mu + \int_{A_2(z)} f_z d\mu
 = 2 \int_{A_1(z)} f_z d\mu,
\]
where the last equality is due to the symmetry of $f_z$ with respect to its last two arguments.
Recall that $|q(|x-z|, r_1, r_2)| \le h(r_1)$ by Lemma~\ref{the:H2Bound}. Now
\begin{align*}
 \int_{A_1(z)} f_z d\mu
 &\le \iiint \ind_{(|z|/6,\infty)}(r_1) \, |B_{r_1}(o) \cap B_{r_2}(x)| \, h(r_1)
 \, dx F(dr_1) F(dr_2) \\
 &\le J(z) \sup_{r>|z|/6} h(r),
\end{align*}
where
\begin{align*}
 J(z)
 &= \iiint \ind_{(|z|/6,\infty)}(r_1) \, |B_{r_1}(o) \cap B_{r_2}(x)| \, dx F(dr_1) F(dr_2) \\
 &= |B_1|^2 \left( \int_{\R+} r^d F(dr) \right) \left( \int_{|z|/6}^\infty r^d F(dr) \right).
\end{align*}
\end{proof}

\begin{proof}[Proof of Proposition~\ref{thm1}]
By Proposition~\ref{the:CovThinnedGrainCover}, we can write
\[
 \frac{\kth(z)}{\bar F(|z|) |z|^d}  = \lambda I_1(z) + \lambda^2 (I_2(z) + I_3(z)),
\]
where
\[
 I_1(z) = (\bar F(|z|) |z|^d)^{-1}  \int_\Rp |B_r(o) \cap B_r(z)| h(r)  F(dr),
\]
and where
\begin{align*}
 I_2(z) &= \iiint_{A_z} f_z(x,r_1,r_2) \, dx F(dr_1) F(dr_2), \\
 I_3(z) &= \iiint_{A_z^c} f_z(x,r_1,r_2) \, dx F(dr_1) F(dr_2),
\end{align*}
denote the integrals of the function
\[
 f_z(x,r_1,r_2)
 = |B_{r_1}(o) \cap B_{r_2}(x) | \left( \frac{q(|x - z|, r_1, r_2)}{\bar F(|z|) |z|^d} \right)
\]
over the set
\[
 A_z = \{ (x,r_1,r_2): |x-z| \le 2(r_1+r_2) \}
\]
and its complement, respectively.

The integral $I_1(z) \to 0$ as $|z| \to \infty$ by Lemma~\ref{lemma_s1}, because $h(r) \to 0$ as
$r \to \infty$ by assumption.

We will next show that $I_2(z) \to 0$ as well. We apply Lemma~\ref{lemma_tail_part}, to conclude
that
\[
 |I_2(z)| \le c_2 ( |z|^d \bar F(|z|) )^{-1} \left( \int_{|z|/6}^\infty r^d F(dr) \right) \sup_{r > |z|/6}
 h(r),
\]
where $c_2 = 2|B_1|^2 \int r^d F(dr)$. The right side above tends to zero as $|z| \to \infty$,
because $h(r) \to 0$ as $r \to \infty$, and because the integral on the right side above is
asymptotically equivalent to constant multiple of $|z|^d \bar F(|z|)$ by
Lemma~\ref{lemma_reg_var}.

To analyze the limiting behavior of $I_3(z)$ as $|z| \to \infty$, note that
assumption~\eqref{eq:CovDecoupling} and Lemma~\ref{the:SlowVariation} imply that for any
$x,r_1,r_2$,
\[
 q(|x-z|,r_1,r_2)
 \ \sim \ q_\infty(r_1,r_2) |x-z|^d \bar F(|x-z|)
 \ \sim \ q_\infty(r_1,r_2) |z|^d \bar F(|z|).
\]
By the definition of $A_z$, it thus follows that
\[
  f_z(x,r_1,r_2) \ind_{A_z^c}(x,r_1,r_2) \to q_\infty(r_1,r_2) | B_{r_1}(o) \cap B_{r_2}(x) |
\]
as $|z| \to \infty$. Moreover, by Lemma~\ref{lemma_main_part} there exists a constant $c_3$ such
that
\[
 | f_z(x,r_1,r_2) \ind_{A_z^c}(x,r_1,r_2)| \le c_3 | B_{r_1}(o) \cap B_{r_2}(x) |
\]
for all $x,r_1,r_2$ and all large enough $z$. Because the right side above is integrable with
respect to $dx F(dr_1) F(dr_2)$, Lebesgue's dominated convergence theorem shows that
\begin{align*}
 \lim_{|z| \to \infty} I_3(z)
 &= \iiint q_\infty(r_1,r_2) | B_{r_1}(o) \cap B_{r_2}(x) | \, dx F(dr_1) F(dr_2) \\
 &= |B_1|^2 \iint q_\infty(r_1,r_2) r_1^d r_2^d \, F(dr_1) F(dr_2),
\end{align*}
which completes the proof of Proposition~\ref{thm1}.
\end{proof}

\subsection{Asymptotic two-point correlation}

\begin{proposition}
\label{lemma_pair_cor}
Assume that the radius distribution $F$ follows a power law with tail exponent $\alpha > d$ and
\begin{equation}
 \label{eq:pair_cor_assumption}
 q(|z|, r_1, r_2) \sim q_\infty(r_1, r_2)|z|^d \bar F(|z|)
 \quad \text{as $|z| \to \infty$}.
\end{equation}
Then
\[
 \xith(z)
 \sim c |z|^d \bar F(|z|)
 \quad \text{as $|z| \to \infty$},
\]
where
\[
 c = \frac{\lambda^2}{\lambdath^2}
  \int_\Rp \int_\Rp q_\infty(r_1, r_2) F(dr_1) F(dr_2).
\]
\end{proposition}
\begin{proof}
Using \eqref{eq:TwoPointCorrelation} we can write
\[\frac{\xith(z)}{|z|^d \bar F(|z|)}
 = \frac{\lambda^2}{\lambdath^2} ( I_1(z) + I_2(z) ),
\]
where
\begin{align*}
I_1(z) &= \iint_{A_z} f_z(r_1, r_2) F(dr_1) F(dr_2),\\
I_2(z) &= \iint_{A_z^c} f_z(r_1, r_2) F(dr_1) F(dr_2),
\end{align*}
denote the integrals of
\[
f_z(r_1, r_2) = \frac{q(|z|, r_1, r_2)}{|z|^d \bar F(|z|)}
\]
over the set
\[
 A_z = \{ (r_1,r_2) \in \R_+ \times \R_+: r_1+r_2 > |z|/2 \}
\]
and its complement, respectively.

Observe that $A_z \subset A_1(z) \cup A_2(z)$, where $A_i(z) = \{(r_1,r_2): r_i > |z|/4\}$, and
that $|q| \le 1$ by Lemma~\ref{the:H2Bound}. As a consequence,
\begin{multline*}
 \int_{A_z} |q(|z|, r_1, r_2)| F(dr_1) F(dr_2)
 \le (F \times F)(A_1(z)) + (F \times F)(A_2(z))
 = 2 \bar F (|z|/4),
\end{multline*}
which implies $I_1(z) \to 0$ as $|z| \to \infty$.

Note that $f_z(r_1, r_2)\ind_{A_z^c}(r_1, r_2) \to q_\infty(r_1, r_2)$ by assumption
\eqref{eq:pair_cor_assumption} and the definition of $A_z$. By Lemma~\ref{lemma_main_part},
$f_z(r_1, r_2)\ind_{A_z^c}(r_1, r_2)$ is bounded for large $z$ uniformly on $r_1$ and $r_2$.
Lebesgue's dominated convergence theorem then shows that
\[
\lim_{|z|\to\infty} I_2(z)
= \iint q_\infty(r_1, r_2) F(dr) F(ds).
\]
\end{proof}

\section{Isolated grains retained}
\label{sec:isolated}

In this section we study the thinning where only isolated grains are retained. In the general
framework of Section~\ref{sec:Thinning}, this is achieved by assigning unit weight to every grain,
so that $G_r(dw) = \delta_1(dw)$. For nonrandom equally sized grains this corresponds to the
classical Matérn type~I thinning.

\begin{theorem}
\label{the:isolated}
Assume that the radius distribution $F$ follows a power law with tail exponent $\alpha > d$, so
that $1-F(r) = \ell(r) r^{-\alpha}$ for some slowly varying function $\ell$. Then the thinned
radius distribution is bounded by
\[
 \Fbarth(r) \le \frac{\lambda}{\lambdath} e^{-\lambda |B_1| r^d},
\]
the covariance function of the thinned grain cover decays according to
\[
 \kth(z) \sim \lambda c_{\alpha,d} \pth^2 \ell(|z|) |z|^{-(\alpha-d)}
 \quad \text{as $|z| \to \infty$},
\]
and the two-point correlation function of the thinned grain centers according to
\[
 \xith(z) \sim \lambda c_{\alpha,d} \ell(|z|) |z|^{-(\alpha-d)}
 \quad \text{as $|z| \to \infty$},
\]
where the constant $c_{\alpha,d}$ is given by~\eqref{eq:calphad}.
\end{theorem}

\begin{proof}
Because the weights are deterministic the retention probabilities have simple formulas
\[
 h(r) = h(r, 1)
 = \exp \left( - \lambda |B_1| \int_\Rp (r+s)^d \, F (ds) \right)
\]
and
\[
 h_2(|z|, r_1, r_2)
 = h(r_1) h(r_2) \exp \left( \lambda \int_\Rp | B_{r_1+s}(o) \cap B_{r_2+s}(z) | \, F(ds) \right).
\]
The tail of the thinned radius distribution \eqref{eq:radiusth} is
\[\Fbarth(r) = \frac{\lambda}{\lambdath}\int_r^\infty h(s) F(ds)
\le \frac{\lambda}{\lambdath} h(r)
\le \frac{\lambda}{\lambdath} e^{-\lambda |B_1|r^d}. \]

To show the claim for the covariance and two-point correlation functions, we will use
Proposition~\ref{thm1} and Proposition~\ref{lemma_pair_cor} respectively. For that we need to show
that \eqref{eq:CovDecoupling} holds. By Lemma \ref{lemma_average_intersection} we have for the
average intersection volume in $h_2$ above
\[
 \int_\Rp | B_{r_1+s}(o) \cap B_{r_2+s}(z) | \, F(ds) \sim c |z|^d \bar F(|z|).
\]
Because the right hand side goes to zero as $|z| \to \infty$ we
can use the fact that $\lim_{t \to \infty} (e^t - 1)/t = 1$ to obtain
\eqref{eq:CovDecoupling} with
\[
 q_\infty(r_1, r_2) = h(r_1)h(r_2) \lambda c_{\alpha, d}.
\]

Using Proposition~\ref{thm1} we find that $\kth(z) \sim c_1 |z|^d \bar F(|z|)$.
Using the formula for volume fraction \eqref{eq:VolumeFraction} we also find the constant
$c_1 = \lambda   \pth^2 c_{\alpha, d}$.
Similarly by Proposition~\ref{lemma_pair_cor} we find that $\xith(z) \sim c_2 |z|^d \bar F(|z|)$.
With the help of germ density \eqref{eq:GermDensity} we have
$c_2 = \lambda c_{\alpha, d}$.
\end{proof}

\section{Random grains retained}
\label{sec:random}

Here we assume that each grain in the proposed Boolean model is assigned a random weight
independently of the other grains, according to some continuous distribution function. The
continuity ensures that there will be no tie breaks. Because the shape of the weight distribution
does not affect the retention probabilities considered here, as long as it is continuous, we may
without loss of generality assume that $G_r(dw) = 1_{(0,1)}(w) dw$, the uniform distribution on
$(0,1)$. Note that for nonrandom equally sized grains, this corresponds to the classical \matern
type~II thinning.

\begin{theorem}
\label{the:random}
Assume that the radius distribution $F$ follows a power law with tail exponent $\alpha > d$, so
that $1-F(r) = \ell(r) r^{-\alpha}$ for some slowly varying function $\ell$. Then the thinned
radius distribution decays according to
\[
 \Fbarth(r) \sim (\lambdath |B_1|)^{-1} \frac{\alpha}{\alpha + d} \ell(r) r^{-(\alpha+d)}
 \quad \text{as $r \to \infty$},
\]
the covariance function of the thinned grain cover according to
\[
 \kth(z) \sim c_1 \ell(|z|) |z|^{-(\alpha-d)}
 \quad \text{as $|z| \to \infty$},
\]
and the two-point correlation function of the thinned grain centers according to
\[
 \xith(z) \sim c_2 \ell(|z|) |z|^{-(\alpha-d)}
 \quad \text{as $|z| \to \infty$},
\]
for some $c_1, c_2 \in (0, \infty)$.
\end{theorem}

\begin{proof}
The retention probability of a grain with radius $r$ and weight $w \in (0, 1)$ is
\[
 h(r, w)
 = \exp \left ( -\lambda \int_\Rp G_R[w, \infty) | B_{r+s}(o) | F (ds) \right)
 = \exp ( -\lambda (1-w) b(r) ),
\]
where
\[
 b(r) =  \int_{\Rp} |B_{r+s}(o)| \, F(ds).
\]
The weight-averaged retention probability thus equals
\begin{align*}
 h(r)
 = \int_0^1 e^{ -\lambda (1-w) b(r) } dw
 = \frac{1 - e^{ - \lambda b(r)}}{\lambda b(r)}.
\end{align*}
Because $b(r) \to \infty$ as $r \to \infty$, it follows that $\lim_{r \to \infty} h(r) = 0$. The
first condition of Proposition~\ref{thm1} is thus satisfied.

Note that $b(r) \sim |B_r|$ as $r \to \infty$, which implies that
\[
 h(r) \sim (\lambda b(r))^{-1} \sim (\lambda |B_1| r^d)^{-1}.
\]
By Lemma \ref{lemma:asympt_int}, the thinned radius distribution \eqref{eq:radiusth} is
\[
 \Fbarth(r)
 = \frac{\lambda}{\lambdath}\int_r^\infty h(s) F(ds)
 \sim (\lambdath |B_1|)^{-1} \int_r^\infty s^{-d} F(ds).
\]
Furthermore by Lemma \ref{lemma_reg_var}
\[
 \Fbarth(r) \sim (\lambdath |B_1|)^{-1} \frac{\alpha}{\alpha+d} r^{-d}\bar F(r).
\]

The pair retention probability equals
\begin{align*}
 &h_2(|z|, r_1, r_2, w_1, w_2) \\
 = & h(r_1, w_1) h(r_2, w_2)
 \exp\left( \lambda \int_\Rp G_R[w_1 \vee w_2, \infty) | B_{r_1+s}(o) \cap B_{r_2+s}(z) | F (ds) \right)\\
 = & \exp\Big( -\lambda (1-w_1) b(r_1) -\lambda (1-w_2) b(r_2) + \lambda (1 - w_1 \vee w_2) a_z(r_1, r_2)
 \Big),
\end{align*}
where
\[
 a_z(r_1,r_2) = \int_{\R+} |B_{r_1+s}(o) \cap B_{r_2+s}(z) | \, F(ds).
\]
From this expression we see that the retention covariance function defined in~\eqref{eq:RetentionCovariance}
equals
\[
 q(|z|,r_1,r_2)
 = \int_0^1 \int_0^1 e^{-\lambda b(r_1) (1-w_1)} e^{-\lambda b(r_2) (1-w_2)} \left( e^{\lambda (1- w_1 \vee w_2) a_z(r_1,r_2)} - 1 \right) dw_1 dw_2.
\]
As $|z| \to \infty$, Lemma~\ref{lemma_average_intersection} shows that the term in parentheses above is asymptotically equivalent to
\[
 e^{\lambda (1- w_1 \vee w_2) a_z(r_1,r_2)} - 1
 \ \sim \
 \lambda (1- w_1 \vee w_2) c_{\alpha,d} \bar F(|z|) |z|^d.
\]
With the help of the bound $|e^t - 1| \le (e-1) t$ for $t \in [0,1]$, we may use dominated convergence
to conclude that
\[
 q(|z|,r_1,r_2) \sim q_\infty(r_1,r_2) |z|^d \bar F(|z|) \quad \text{as $|z| \to \infty$},
\]
where
\begin{align*}
 q_\infty(r_1,r_2)
 &= \lambda c_{\alpha,d} \int_0^1 \int_0^1 (1-w_1 \vee w_2) e^{-\lambda b(r_1) (1-w_1)} e^{-\lambda b(r_2) (1-w_2)}
 dw_1 dw_2.
\end{align*}
Now by Proposition~\ref{thm1} it follows that
\[
 \kth(z) \sim c_1 \bar F(|z|) |z|^d \quad \text{as $|z| \to \infty$},
\]
where
\[
 c_1 = \lambda^2 |B_1|^2 \int_{\R+} \int_{\R_+} r_1^d r_2^d q_\infty(r_1,r_2) F(dr_1) F(dr_2).
\]
The constant $c_1$ is finite because $q_\infty(r_1,r_2) \le \lambda c_{\alpha,d}$ for all $r_1,r_2$. The fact that $c_1$ is strictly positive
is easily seen by inspecting the expression of $q_\infty(r_1,r_2)$.

Similarly, Proposition~\ref{lemma_pair_cor} shows that
\[
 \xith(z) \sim c_2 |z|^d \bar F(|z|),
\]
where
\[
 c_2 = \frac{\lambda^2}{\lambdath^2} \int_\Rp \int_\Rp q_\infty(r_1,r_2) F(dr_1) F(dr_2).
\]
The finiteness and strict positivity of $c_2$ follow by similar reasoning as for $c_1$.
\end{proof}

\section{Large grains retained}
\label{sec:large}

A thinning which favors large grains is obtained by letting the weight of each grain be equal to
its radius, so that $G_r(dw) = \delta_r(dw)$.

\begin{theorem}
\label{the:large}
Assume that the radius distribution $F$ follows a power law with tail exponent $\alpha > d$, so
that $1-F(r) = \ell(r) r^{-\alpha}$ for some slowly varying function $\ell$. Then the thinned
radius distribution decays according to
\[
 \Fbarth(r) \sim \frac{\lambda}{\lambdath} \ell(r) r^{-\alpha}
 \quad \text{as $r \to \infty$},
\]
the covariance function of the thinned grain cover according to
\[
 \kth(z) \sim \lambda c_{\alpha,d} (1-\pth)^2 \ell(|z|) |z|^{-(\alpha-d)}
 \quad \text{as $|z| \to \infty$},
\]
and the two-point correlation function of the thinned grain centers according to
\[
 \xith(z) \sim \lambda c_{\alpha,d} \ell(|z|) |z|^{-(\alpha-d)}
 \quad \text{as $|z| \to \infty$},
\]
where the constant $c_{\alpha,d}$ is given by~\eqref{eq:calphad}.
\end{theorem}

\begin{proof}
Because the weight of each grain is equal to its radius, the weight-averaged retention probability
$h(r)$ is equal to $h(r,w)$ with $w$ taking on the value $r$. By Proposition
\ref{the:RetentionProbability}, the retention probability is given by
\[
 h(r) = \exp \left( - \lambda \int_{\R_+} | B_{r+s} | \, \I_{[r,\infty)}(s) \, F(ds) \right).
\]
Because the integrand above tends to zero as $r \to \infty$, and the integrand is bounded by the
$F(ds)$-integrable function $|B_{2s}|$, dominated convergence implies that $\lim_{r \to \infty}
h(r) = 1$. By Lemma \ref{lemma:asympt_int}, the tail of the thinned radius distribution
\eqref{eq:radiusth} satisfies
\[
 \Fbarth(r)
 = \frac{\lambda}{\lambdath} \int_r^\infty h(s) F(ds)
 \sim \frac{\lambda}{\lambdath} \int_r^\infty F(ds)
 = \frac{\lambda}{\lambdath} \bar F(r).
\]

To analyze the long-range behavior of $\kth(z)$ and $\xith(z)$, let us first investigate the
long-range behavior of the retention covariance function $q(|z|,r_1,r_2)$ defined
by~\eqref{eq:RetentionCovariance}. Using Proposition~\ref{the:H2}, we find that
\begin{equation}
 \label{eq:RetentionLarge}
 q(|z|,r_1,r_2) = h(r_1) h(r_2) \left( \ind(r_1+r_2 < |z|) e^{\tau(|z|,r_1,r_2)} - 1 \right),
\end{equation}
where
\begin{equation}
 \label{eq:TauLarge}
 \tau(|z|,r_1,r_2)
 = \lambda \int_{\R_+} \ind_{[r_1 \vee r_2, \infty)}(s) | B_{r_1+s}(o) \cap B_{r_2 + s}(z) | \, F(ds).
\end{equation}
When $|z| > 3(r_1+r_2)$, we may replace the region of integration above with the full positive
real line, so that with the help of Lemma \ref{lemma_average_intersection} we find that
\[
 \tau(|z|,r_1,r_2)
 \, = \, \lambda \int_{\R+} | B_{r_1+s}(o) \cap B_{r_2 + s}(z) | \, F(ds)
 \, \sim \, \lambda c_{\alpha,d} |z|^d \bar F(|z|),
\]
as $|z| \to \infty$. Because $e^t - 1 \sim t$ for small $t$, we conclude
using~\eqref{eq:RetentionLarge} that
\begin{equation}
 \label{eq:large_1}
 q(|z|,r_1,r_2) \sim q_\infty(r_1,r_2) |z|^d \bar F(|z|),
\end{equation}
where
\begin{equation}
 \label{eq:QInftyLarge}
 q_\infty(r_1,r_2) = \lambda c_{\alpha, d} h(r_1)h(r_2).
\end{equation}
The claim for the two-point correlation function $\xith(z)$ now follows by using
Proposition~\ref{lemma_pair_cor}, after noting that the constant in
Proposition~\ref{lemma_pair_cor} is
\[
 \frac{\lambda^2}{\lambdath^2} \iint q_\infty(r_1, r_2) F(dr_1) F(dr_2)
 = \lambda\lambdath^{-2} c_{\alpha, d} \left( \lambda \int h(r) F(dr)\right)^2
 =  \lambda c_{\alpha, d}.
\]

We will now move on to the part concerning the covariance function $\kth(z)$ of the thinned grain
cover. Note that because $h(r)$ does not vanish as $r \to \infty$, we cannot use
Proposition~\ref{thm1} to deduce the long-range behavior of $\kth(z)$. Instead, we will proceed by
directly analyzing the integral building blocks of $\kth(z)$ in high precision. Let us start by
rewriting \eqref{eq:cov} as
\[
 \kth(z) = \lambda I_0(z) + \lambda^2 (I_1(z) + I_2(z) + I_3(z)),
\]
where
\[
 I_0(z) = \int_\Rp |B_r(o) \cap B_r(z)| \, h(r) \, F(dr),
\]
\[
 I_j(z) = \iiint_{A_j^z} |B_{r_1}(o) \cap B_{r_2}(x)| q(|x - z|, r_1, r_2) \, dx F(dr_1) F(dr_2),
 \quad j=1,2,3,
\]
and
\begin{align*}
 A_1^z &= \{(x, r_1, r_2) : |x| < r_1+r_2, \ r_1+r_2 < |x-z|/2 \},\\
 A_2^z &= \{(x, r_1, r_2) : |x| < r_1+r_2, \ |x-z|/2 < r_1+r_2 < |x-z|\},\\
 A_3^z &= \{(x, r_1, r_2) : |x| < r_1+r_2, \ |x-z| < r_1+r_2 \}.
\end{align*}

The first term $ I_0(z) \lesssim c_{\alpha,d} |z|^d \bar F(|z|) $
by Lemma~\ref{lemma_average_intersection}.
Note that the integrand in $I_0(z)$ vanishes for $r \le |z|/2$ so that
\begin{align*}
 I_0(z) \ge \inf_{r \ge |z|/2} h(r) \int_\Rp |B_r(o) \cap B_r(z)| F(dr).
\end{align*}
Using Lemma~\ref{lemma_average_intersection} and
that $h(r) \to 1$ as $r \to \infty$ we conclude that
\begin{equation}
 \label{eq:cov_1}
 I_0(z) \sim c_{\alpha,d} |z|^d \bar F(|z|).
\end{equation}

Next, we will prove that
\begin{equation}
 \label{eq:cov_2}
 I_1(z) \sim \lambda^{-1} c_{\alpha,d} \pth^2 |z|^d \bar F(|z|).
\end{equation}
By Lemma \ref{lemma_main_part}, the function $(x,r_1,r_2) \mapsto \frac{q(|x-z|, r_1, r_2)}{|z|^d
\bar F(|z|)} \I_{A_1^z}(x,r_1,r_2)$ is positive and bounded by a constant which does not depend on
$z$. Because $|B_{r_1}(o) \cap B_{r_2}(x)|$ is integrable with respect to $dx F(dr_1) F(dr_2)$,
Lebesgue's dominated convergence theorem shows that
\[
 \lim_{z \to \infty} \frac{I_1(z)}{|z|^d \bar F(|z|)}
 = \iiint |B_{r_1}(o) \cap B_{r_2}(x)|
 \left( \lim_{z \to \infty} \frac{q(|x - z|, r_1, r_2)}{|z|^d \bar F(|z|)} \I_{A_1^z} \right)
 dx F(dr_1) F(dr_2).
\]
Using \eqref{eq:large_1} and the definition of $A_1^z$, the limit on the right equals
$q_\infty(r_1, r_2)$. Plugging in the expression~\eqref{eq:QInftyLarge} for $q_\infty(r_1,r_2)$
and recalling the formula~\eqref{eq:VolumeFraction} for the volume fraction of the thinned grain
cover $\pth$, we find that
\begin{align*}
 \lim_{|z| \to \infty} \frac{I_1(z)}{|z|^d \bar F(|z|)}
 &= \lambda c_{\alpha,d} \iiint |B_{r_1}(o) \cap B_{r_2}(x)| h(r_1) h(r_2) \, dx F(dr_1) F(dr_2)\\
 &= \lambda c_{\alpha,d} \iint |B_1|^2 r_1^d r_2^d h(r_1) h(r_2) \, F(dr_1) F(dr_2)\\
 &= \lambda^{-1} c_{\alpha,d} \pth^2,
\end{align*}
which proves the validity of~\eqref{eq:cov_2}.

Now we will prove that
\begin{equation}
 \label{eq:cov_I_2}
 \frac{I_2(z)}{|z|^d \bar F(|z|)} \to 0
 \quad \text{as $|z| \to \infty$}.
\end{equation}
First, using the bound $|B_{r_1+s}(o) \cap |B_{r_2+s}(x-z)| \le |B_{r_1+s}| \le |B_{2s}|$ for $s
\ge r_1 \vee r_2$, we find that the function $\tau$ defined in~\eqref{eq:TauLarge} is bounded by
\[
 \tau(|x-z|,r_1,r_2)
 \le \lambda |B_1| 2^d \int \ind_{[r_1 \vee r_2, \infty)}(s) s^d \, F(ds).
\]
Observe next that
\[
 |z| \le 3(r_1+r_2) \le 6 (r_1 \vee r_2)
\]
for all $(x,r_1,r_2) \in A_2^z$, so that $r_1 \vee r_2$ is large when $|z|$ is large. As a
consequence, we see by Lemma~\ref{lemma_reg_var} and Lemma~\ref{the:UniformConvergence} that
for all $(x,r_1,r_2) \in A_2^z$ and all large enough $z$,
\begin{align*}
 \tau(|x-z|, r_1, r_2)
 &\le 2 \lambda |B_1| 2^d (r_1 \vee r_2)^d \bar F(r_1 \vee r_2)\\
 &\le 4 \lambda |B_1| 2^d (|z|/6)^d \bar F(|z|/6).
\end{align*}
Because $e^t - 1 \le (e-1)t$ for $t \in [0,1]$, formula~\eqref{eq:RetentionLarge} combined with
the above inequality shows that for all $(x,r_1,r_2) \in A_2^z$ and all large enough $z$,
\[
 0 \le q(|x-z|, r_1, r_2)
 \le c_1 |z|^d \bar F(|z|/6),
\]
where $c_1 = 4 (e-1) \lambda |B_1| 3^{-d}$. Therefore,
\[
 0
 \le I_2(z)
 \le c_1 |z|^d \bar F(|z|/6) \iiint_{A_2^z} |B_{r_1}(o) \cap B_{r_2}(x)| \, dx F(dr_1) F(dr_2).
\]
Note that $A_2^z \subset A_{21}^z \cup A_{22}^z$ where $A_{2i}^z = \{(x,r_1,r_2): r_i \ge |z|/6\}$,
$i=1,2$. By symmetry of the integrand with respect to $r_1$ and $r_2$,
\begin{align*}
 0 \le I_2(z)
 &\le 2 c_1 |z|^d \bar F(|z|/6) \int_\Rp \int_{|z|/6}^{\infty} \int_{\R^d} |B_{r_1}(o) \cap B_{r_2}(x)| \, dx F(dr_1) F(dr_2)\\
 &= 2 c_1 |z|^d \bar F(|z|/6) |B_1|^2 \left( \int r^d \, F(dr) \right)
 \left(  \int_{|z|/6}^{\infty} r^d \, F(dr) \right),
\end{align*}
which shows the validity of~\eqref{eq:cov_I_2}.

It remains to be shown that
\begin{equation}
 \label{eq:cov_3}
 I_3(z) \sim -2 \frac{\pth}{\lambda} c_{\alpha,d} |z|^d \bar F(|z|).
\end{equation}
To do that, we first fix a small $\epsilon \in (0,1/4)$. Note that by
formula~\eqref{eq:RetentionLarge}, the retention covariance function equals $q(|x-z|,r_1,r_2) =
-h(r_1)h(r_2)$ for $(x,r_1,r_2) \in A_3^z$.
Note also that, for fixed $r_1$ and $r_2$ the $x$-slice of $A_3^z$ is
\[
\{ x : (x, r_1, r_2) \in A_3^z \} = B_{r_1+r_2}(z) \cap B_{r_1+r_2}(o).
\]
Because $|B_{r_1}(o) \cap B_{r_2}(x)|$ vanishes for $x$ outside $B_{r_1+r_2}(o)$, we may represent
$I_3(z)$ according to
\begin{align*}
 I_3(z)
 &= - \iiint_{A_3^z} |B_{r_1}(o) \cap B_{r_2}(x)| \, h(r_1)h(r_2) \, dx  F(dr_1) F(dr_2)\\
 &= - \iint_{C_3^z} \int_{B_{r_1+r_2}(z)} |B_{r_1}(o) \cap B_{r_2}(x)| \, dx \, h(r_1)h(r_2) \, F(dr_1) F(dr_2)\\
\end{align*}
where
\[
 C_3^z = \{ (r_1, r_2) : |z| \le 2(r_1+r_2) \}.
\]
Next we split $I_3(z)$ into three parts
\[
I_3(z)= - (I_{31}(z) + I_{32}(z) + I_{33}(z)),
\]
where
\[
I_{3j}(z) = \iint_{A_{3j}^z} \int_{B_{r_1+r_2}(z)} |B_{r_1}(o) \cap B_{r_2}(x)| \, dx \, h(r_1)h(r_2) \, F(dr_1) F(dr_2)
\]
for $j = 1, 2, 3$ and
\begin{align*}
 C_{31}^z    &= C_{311}^z \cup C_{312}^z \\
 C_{311}^z   &= \{(r_1, r_2) : 0 \le r_1 \le \epsilon |z|, \ |z|/2 \le r_2 \}\\
 C_{312}^z   &= \{(r_1, r_2) : 0 \le r_2 \le \epsilon |z|, \ |z|/2 \le r_1 \}\\
 C_{32}^z    &= C_{331}^z \cup C_{332}^z\\
 C_{321}^z   &= \{ (r_1, r_2) : 0 \le r_1 \le \epsilon |z|, \ |z|/2 - r_1 \le r_2 \le |z|/2 \}\\
 C_{322}^z   &= \{ (r_1, r_2) : 0 \le r_2 \le \epsilon |z|, \ |z|/2 - r_2 \le r_1 \le |z|/2 \}\\
 C_{33}^z    &= [\epsilon |z|, \infty)^2 \cap C_3^z.
\end{align*}

A change of variables shows that
\[
 \int_{B_{r_1+r_2}(z)} |B_{r_1}(o) \cap B_{r_2}(x)| \, dx
 = \int_{B_{r_1}(o)} |B_{r_2}(z) \cap B_{r_1+r_2}(x) | \, dx,
\]
so that we can express the integral $I_{31}(z)$ more conveniently as
\[
 I_{31}(z)
 = \iint_{C_{31}^z} \int_{B_{r_1}(o)} |B_{r_2}(z) \cap B_{r_1+r_2}(x) | \, dx \, h(r_1)h(r_2) \, F(dr_1)
 F(dr_2).
\]
By symmetry we can write $I_{31}(z) = 2 I_{311}(z)$, where $I_{311}(z)$ is a modification of
$I_{31}(z)$ with the region of integration $C_{31}^z$ replaced by $C_{311}^z$. To analyze the
long-range behavior of $I_{311}(z)$, let us split it according to $I_{311}(z) = J_{1}(z) +
J_{2}(z)$, where
\[
 J_1(z)
 = \int_0^{\epsilon |z|} \int_{|z|/2}^\infty \int_{B_{r_1}(o)}
 |B_{r_2}(z) \cap B_{r_2}(o)| \, dx \, h(r_2) F(dr_2) h(r_1) F(dr_1),
\]
and where $J_2(z) = I_{311}(z) - J_1(z)$. Because the integrand of $J_1(z)$ does not depend on
$x$, we can rewrite the integral as
\[
 J_1(z)
 = \left( \int_0^{\epsilon |z|} |B_r| \, h(r) F(dr) \right)
 \left( \int_{|z|/2}^\infty|B_{r}(z) \cap B_{r}(o)| \, h(r) F(dr) \right).
\]
The first integral on the right satisfies
\[
 \int_{0}^{|z|\epsilon} |B_1| r^d h(r) F(dr) \sim \int_\Rp |B_1| r^d h(r) F(dr)
 = \lambda^{-1} \pth,
\]
where $\pth$ is the volume fraction of the thinned grain cover given by~\eqref{eq:VolumeFraction}.
Notice that, because the intersection in the second integral vanishes for $r < |z|/2$,
we can apply \eqref{eq:cov_1} to conclude that
\begin{equation*}
 J_1(z) \sim \lambda^{-1} \pth c_{\alpha,d} |z|^d \bar F(|z|).
\end{equation*}

The rest of the proof constitutes of showing that the remaining three parts of $I_3(z)$ are
negligible. We start by showing that $J_2 \ge 0$ and
\[
 \limsup_{|z| \to \infty} \frac{J_2(z)}{|z|^d \bar F(|z|)}
 \le |B_1|^2 \left((1+2\epsilon)^d - 1\right) \left( \int r^d F(dr) \right)
 \frac{\alpha}{\alpha-d} 2^{\alpha-d}.
\]
First we need a bound for the difference of the intersections in $J_2(z)$. Fix $x \in
B_{r_1}(o)$ and $(r_1,r_2) \in C_{311}^z$. Because $|x| \le r_1$, we have $B_{r_2}(o) \subset
B_{r_1+r_2}(x)$, which implies that the integrand in $J_2(z)$ is bounded by
\begin{align*}
 0
 &\le |B_{r_2}(z) \cap B_{r_1+r_2}(x)| - |B_{r_2}(z) \cap B_{r_2}(o)|\\
 &=   |B_{r_2}(z) \cap (B_{r_1+r_2}(x) \setminus  B_{r_2}(o))|\\
 &\le |B_{r_1+r_2}(x)| - | B_{r_2}(o)|\\
 &=   |B_1| ((1+r_1/r_2)^d - 1) r_2^d \\
 &\le |B_1| ((1 + 2\epsilon)^d - 1) r_2^d,
\end{align*}
where the last inequality is due to $r_1 \le \epsilon |z|$ and $|z|/2 \le r_2$. This bound and
$h(r) \le 1$ now imply that
\begin{align*}
 0 \le J_2(z)
 &\le \int_0^{\epsilon |z|} \int_{|z|/2}^\infty \int_{B_{r_1}(o)}
 |B_1| ((1 + 2\epsilon)^d - 1 ) r_2^d \, dx h(r_2) F(dr_2) h(r_1) F(dr_1) \\
 &\le |B_1|^2 ((1 + 2\epsilon)^d - 1 ) \left( \int_0^{\epsilon |z|} r^d F(dr) \right) \ \left( \int_{|z|/2}^\infty r^d  F(dr) \right).
\end{align*}
Now using Lemma~\ref{lemma_reg_var} proves the claim.

We will now show that $I_{32} \ge 0$ and
\[
 \limsup_{|z| \to \infty} \frac{I_{32}(z)}{|z|^d \bar F(|z|)} \le
 2\left( \int r^d F(dr) \right)|B_1|^2
 \left[(1-2\epsilon)^{d-\alpha} - 1 \right]
 \frac{\alpha}{\alpha - d} 2^{\alpha - d}.
\]
By symmetry $I_{32}(z) = 2 I_{321}(z)$, where
\[
 I_{321}(z) = \iint_{C_{321}^z} \int_{B_{r_1+r_2}(z)} |B_{r_1}(o) \cap B_{r_2}(x)| \, dx \,
 h(r_1)h(r_2) \, F(dr_1) F(dr_2).
\]
Note that $C_{321}^z \subset [0, \epsilon |z|] \times [|z|/2(1-2\epsilon), |z|/2]$.
Also approximating $B_{r_1 + r_2}(z)$ by $\R^d$ and recalling that $h(r) \le 1$ we have
\begin{align*}
 I_{321}(z)
 \le&
 \int_{0}^{|z|\epsilon}
 \int_{(|z|/2)(1-2\epsilon)}^{|z|/2}
 |B_1|^2 r_1^d r_2^d  F(dr_2) F(dr_1)\\
 =&|B_1|^2 \left( \int_{0}^{|z|\epsilon} r^d F(dr) \right)
 \left( \int_{(|z|/2)(1-2\epsilon)}^{\infty} r^d
 F(dr) - \int_{|z|/2}^{\infty} r^d F(dr)
 \right).
\end{align*}
Now using Lemma~\ref{lemma_reg_var} implies the claim.

For the last part $I_{33}(z)$ we have first a simple bound
\begin{align*}
 I_{33}(z)
 &\le \int_{\epsilon |z|}^\infty \int_{\epsilon |z|}^\infty \int_{B_{r_1+r_2}(z)} |B_{r_1}(o) \cap B_{r_2}(x)|
 \, dx \, h(r_1)h(r_2) \, F(dr_1) F(dr_2)\\
 &\le \int_{\epsilon |z|}^\infty \int_{\epsilon |z|}^\infty |B_1|^2 \, r_1^d r_2^d \, F(dr_1) F(dr_2)\\
 &= |B_1|^2 \left(\int_{\epsilon |z|}^\infty r^d F(dr)\right)^2.
\end{align*}
Using Lemma~\ref{lemma_reg_var} this bound implies that
\[
 \frac{I_{33}(z)}{|z|^d \bar F(z)} \to 0 \quad \text{as } |z| \to \infty.
\]

Adding together all the parts of $I_{3}(z)$ we have
\[\limsup_{|z| \to \infty} \frac{I_3(z)}{|z|^d \bar F(|z|)} = -2 \frac{\pth}{\lambda}c_{\alpha, d}\]
and
\[ \liminf_{|z| \to \infty} \frac{I_3(z)}{|z|^d \bar F(|z|)} \ge -2 \frac{\pth}{\lambda}c_{\alpha, d} - \delta(\epsilon), \]
where
\[ \delta(\epsilon) = 2|B_1|^2 \left( \int r^d F(dr) \right)
\frac{\alpha}{\alpha-d} 2^{\alpha-d}
\left(
\left((1+2\epsilon)^d - 1\right)
+
\left((1-2\epsilon)^{d-\alpha} - 1 \right)
\right).
\]
Letting $\epsilon \to 0$ shows the validity of \eqref{eq:cov_3} and concludes the proof.
\end{proof}

\section{Small grains retained}
\label{sec:small}

In this section we study a thinning which favors small grains. This thinning is obtained by
setting the weight of each grain to the inverse of its radius, so that $G_r(dw) =
\delta_{1/r}(dw)$. The following theorem shows that the thinned radius distribution and the key
second-order characteristics decay rapidly to zero, regardless of the tail behavior of the
original radius distribution $F$. Note that here, unlike in
Theorems~\ref{the:isolated}--\ref{the:large}, there is no need to assume anything on the shape of
the radius distribution $F$.

\begin{theorem}
\label{the:small}
Assume that the radius distribution $F$ satisfies $\int r^d F(dr) < \infty$. Then the thinned
radius distribution is bounded by
\[
  \Fbarth(r) \le \frac{\lambda}{\lambdath} e^{-\lambda |B_1| \frac{1}{2} r^d},
\]
the covariance function of the thinned grain cover by
\[
 |\kth(z)| \le e^{- \lambda |B_1| c |z|^d },
\]
and the two-point correlation function of the thinned grain centers by
\[
 |\xith(z)| \le e^{- \lambda |B_1| c |z|^d }
\]
for all large values of $r$ and $z$, where $c \in (0,\infty)$.
\end{theorem}

\begin{proof}
Now the weight-averaged retention probability $h(r)$ is equal to the retention probability of a
reference grain with radius $r$ and weight $1/r$. Using Proposition~\ref{the:RetentionProbability}
we find that
\[
 h(r) = \exp \left( - \lambda \int_0^r |B_{r+s}(o)| F (ds) \right).
\]
From this expression we see that $h$ decreases monotonically to zero as $r$ grows, and that $h(r)
\le \exp( - \frac12 \lambda |B_1| r^d)$ for all large enough $r$ so that $F(r) \ge 1/2$.
Proposition~\ref{the:H2} further shows that the weight-averaged pair retention probability equals
\begin{equation}
 \label{eq:SmallGrainH2}
 h_2(|z|, r_1, r_2) = h(r_1) h(r_2)
 \exp \left( \lambda \int \ind_{[0,r_1 \wedge r_2]}(r) | B_{r_1+r}(o) \cap B_{r_2+r}(z) | F(dr) \right)
\end{equation}
for $|z| > r_1 + r_2$.

To analyze the covariance function of the thinned grain cover, recall that
\begin{multline}
 \label{eq:SmallGrainCov}
 \kth(z) = \lambda \int |B_r(o) \cap B_r(z)| h(r) F(dr) \\
 + \lambda^2 \iiint |B_{r_1}(o) \cap B_{r_2}(x)| q(|x-z|,r_1,r_2) \, dx F(dr_1) F(dr_2),
\end{multline}
where $q(u,r_1,r_2) = h_2(u,r_1,r_2) - h(r_1) h(r_2)$. Because $|B_r(o) \cap B_r(z)| \le |B_1| r^d
\ind(r > |z|/2)$, the first term on the right side of~\eqref{eq:SmallGrainCov} is bounded from
above by
\[
 \lambda |B_1| \left( \int_{|z|/2}^\infty r^d F(dr) \right) \sup_{r > |z|/2} h(r).
\]
Note that $q(|x-z|,r_1,r_2)$ vanishes for $|x-z| > 2(r_1+r_2)$, because the integral
in~\eqref{eq:SmallGrainH2} vanishes for $|z| \ge 2(r_1+r_2)$. This is why the integration in
second term in~\eqref{eq:SmallGrainCov} can be restricted to the set $A(z) = \{(x,r_1,r_2): |x-z|
\le 2(r_1+r_2)\}$. Now using Lemma~\ref{lemma_tail_part}, the absolute value of
second term in~\eqref{eq:SmallGrainCov} is bounded from above by
\[
 2 \lambda^2 |B_1|^2 \left( \int_{\R+} r^d F(dr) \right)
 \left( \int_{|z|/6}^\infty r^d F(dr) \right) \sup_{r>|z|/6} h(r).
\]
As consequence, we find that
\[
 |\kth(z)| \le \left( \lambda m_1 + 2 \lambda^2 m_1^2 \right) \sup_{r>|z|/6} h(r),
\]
where $m_1 = |B_1| \int r^d F(dr)$ is the mean volume of a grain. Therefore,
\[
 |\kth(z)|
 \le \left( \lambda m_1 + 2 \lambda^2 m_1^2 \right) e^{-\frac12 \lambda |B_1| (|z|/6)^d}
\]
for all large enough $z$ such that $F(|z|) \ge 1/2$. A similar analysis can be carried out for the
two-point correlation function.

Typical radius has tail probabilities \eqref{eq:radiusth}
\[
 \Fbarth(r)
 =   \frac{\lambda}{\lambdath} \int_r^\infty h(s) \, F(ds)
 \le \frac{\lambda}{\lambdath} h(r) \bar F(r)
 \le \frac{\lambda}{\lambdath} e^{- \frac{1}{2} \lambda |B_1| r^d}
\]
for all large enough $r$ so that $F(r) \ge 1/2$.
\end{proof}

\section{Conclusions and future work}
\label{sec:Conclusions}

Boolean models consisting of randomly sized spheres in $\R^d$ are long-range dependent if the
sphere radii follow a power-law distribution with tail exponent $\alpha \in (d,2d)$. We studied
second-order statistical properties of four hard-core germ--grain models which are obtained from
such Boolean models using a natural weight-based thinning mechanism. We found that a thinning
which favors large grains preserves the power-law covariance decay of the proposed Boolean model,
whereas a thinning which favors small grains does not. The power-law covariance decay is also
preserved under the thinning where only isolated grains are retained (Matérn type I), and the
thinning where retention is determined by independent weights (Matérn type II). The germ--grain
model obtained by the Matérn type I thinning is an interesting example of a homogeneous
hard-sphere model where typical spheres have exponentially small sizes but the covariance function
decays slowly according to a power law.

To keep the notation simple and the paper easy to read, the analysis carried out in this article
was restricted to spherical grains. However, we believe that this assumption can be easily relaxed
to some extent following standard techniques of stochastic geometry. Another interesting open
problem is to investigate how thinnings affect covariance decay properties in the light-tailed
setting where the grain size distribution in the proposed Boolean model is assumed to decay
rapidly.

\appendix

\section{Regular variation}
\label{sec:RegularVariation}

A measurable function $f: \R_+ \to \R$ is called \emph{regularly varying} with exponent $\gamma
\in \R$ if it is positive for all large enough input values and for all $a > 0$,
\[
 \frac{f(at)}{f(t)} \to a^\gamma
\]
as $t \to \infty$. A regularly varying function with exponent zero is called \emph{slowly
varying}. For a good overview on the theory of regular variation, see for example
\cite{Bingham_Goldie_Teugels_1987}. In this section we will summarize some key properties of
regularly varying functions which are needed in the analysis. The first one is a consequence of
the Karamata's theorem \cite{Bingham_Goldie_Teugels_1987}.

\begin{lemma}
\label{lemma_reg_var}
Let $F$ be distribution function on $\R_+$ with a regularly varying tail of
exponent $\alpha > p$. Then for any constant $a > 0$,
\[
 \int_{a x}^\infty r^p F(dr)
 \sim \left(\frac{\alpha}{\alpha - p}\right) a^{-(\alpha-p)} \bar F (x) x^p
 \quad \text{as $x \to \infty$}.
\]
\end{lemma}

\begin{lemma}
\label{the:SlowVariation}
Assume that $\ell$ is slowly varying. Then for any $z_0 \in \R^d$,
\[
 \ell( |z-z_0| ) \sim \ell( |z| )
 \quad \text{as $|z| \to \infty$}.
\]
\end{lemma}
\begin{proof}
Fix $z_0 \in \R^d$, and write $|z-z_0| = a_z |z|$, where $a_z = |z-z_0|/|z|$. Because $a_z \to 1$
as $|z| \to \infty$, we can fix $m$ such that $a_z \in [1/2,3/2]$ for $|z| > m$. Now, for any $z
\in \R^d$ such that $|z| > m$, it follows that
\[
 \left| \frac{\ell(|z-z_0|)}{\ell(|z|)} - 1 \right|
 =   \left| \frac{\ell(a_z |z|)}{\ell(|z|)} - 1 \right|
 \le \sup_{a \in [1/2,3/2]} \left| \frac{\ell(a |z|)}{\ell(|z|)} - 1 \right|.
\]
The right side above tends to zero as $|z| \to \infty$ because $\ell(a|z|)/\ell(|z|) \to 1$
locally uniformly with respect to $a$ \cite[Thm.~1.2.1]{Bingham_Goldie_Teugels_1987}.
\end{proof}

\begin{lemma}
\label{the:UniformConvergence}
Assume that $f$ is regularly varying with exponent $-\gamma < 0$. Then there exists a constant $u
> 0$ such that $f(y) \le 2f(x)$ for all $y \ge x \ge u$.
\end{lemma}
\begin{proof}
By the uniform convergence of regularly varying functions
\cite[Thm.~1.5.2]{Bingham_Goldie_Teugels_1987}, $f(\lambda x)/f(x) \to \lambda^{-\gamma}$
uniformly for $\lambda \ge 1$ as $x \to \infty$. This implies that we can find $u > 0$ such that
$f(\lambda x) \le 2f(x)$ for all $x \ge u$ and all $\lambda \ge 1$. Now because $y \ge x$, we have
\[
 f(y) \le \sup_{\lambda \ge 1}f(\lambda x) \le 2f(x).
\]
\end{proof}

\begin{lemma}
\label{lemma:asympt_int} Let $F$ a probability measure on $\R_+$, and let $f$ and $g$ be bounded
positive functions on $\Rp$ such that $f(r) \sim g(r)$ as $r \to \infty$. Then,
\[
 \int_r^\infty f(s) F(ds) \sim \int_r^\infty g(s) F(ds)
 \quad \text{as $r \to \infty$}.
\]
\begin{proof}
\begin{align*}
\left| 1 - \frac{\int_r^\infty f(s) F(ds)} {\int_r^\infty g(s) F(ds)} \right| = \left|
\frac{\int_r^\infty  \left(1 - \frac{ f(s) }{ g(s) } \right) g(s) F(ds)} {\int_r^\infty g(s)
F(ds)} \right| \leq \sup_{r \leq s} \left| 1 - \frac{ f(s) }{ g(s) } \right|
\end{align*}
\end{proof}
\end{lemma}

\section{Intersections of distant balls}
\label{sec:Intersections}

\begin{lemma}
\label{lemma_average_intersection}
Let $F$ be a probability distribution on $\Rp$ which follows a power law with tail exponent
$\alpha > d$. Then for any $r_1, r_2 \geq 0$,
\begin{equation}
\label{eq:ai2}
 \int_\Rp | B_{r_1 + r}(o) \cap B_{r_2 + r}(z) | \, F(dr)
 \sim c_{\alpha,d} \, \bar F(|z|) |z|^d
 \quad \text{as $|z| \to \infty$},
\end{equation}
where
\begin{equation}
 \label{eq:calphad}
 c_{\alpha,d} = \int_0^\infty | B_r(o) \cap B_r(e_1) | \alpha r^{-\alpha-1} dr,
\end{equation}
and $e_1$ is the first unit vector in the standard basis of $\R^d$.
\end{lemma}
\begin{proof}
Because the Lebesgue measure is rotation-invariant, we may assume without loss of generality that
$z = ue_1$ for $u > 0$. Fix $r_1,r_2 \ge 0$, and denote left side of \eqref{eq:ai2} by $I(u)$. We
will prove the claim by first showing that
\begin{equation}
 \label{eq:I1Asymptotic}
 I_1(u) = \int_\Rp | B_{r}(o) \cap B_{r}(ue_1) | \, F(dr) \sim c u^d \bar F(u),
\end{equation}
and then showing that the remainder $I_2(u) = I(u) - I_1(u)$ tends to zero faster than $u^d \bar
F(u)$ as $u \to \infty$.

To prove \eqref{eq:I1Asymptotic}, let $F_u$ be the distribution of a random variable obtained by
dividing a $F$-distributed random variable by $u$, so that $F_u(r) = F(ur)$. Then a change of
variables shows that
\[
 I_1(u) = u^d \int_\Rp | B_{r/u}(o) \cap B_{r/u}(e_1) | \, F(dr)
 = u^d \int_\Rp \phi(r) \, F_u(dr),
\]
where $\phi(r) = |B_r(0) \cap B_r(e_1)|$. Because $\phi$ is continuous, $r^{-d} \phi(r) \le |B_1|$
for all $r>0$, and $\phi(r) = 0$ for $r \le 1/2$, we may apply \cite[Lemma
2]{Kaj_Leskela_Norros_Schmidt_2007} (with $p=d$, $\gamma = \alpha$, $q=\alpha+1$) to obtain
\[
 \int_\Rp \phi(r) \, F_u(dr)
 \sim \bar F(u) \int_0^\infty \phi(r) \, \alpha r^{-\alpha-1} dr
 = c \bar F(u),
\]
which implies the validity of \eqref{eq:I1Asymptotic}.

To show that $I_2(u)$ tends to zero faster than $u^d \bar F(u)$, note first that for all $u >
2(r_1 + r_2)$,
\begin{equation}
 \label{eq:I2Integral}
 I_2(u)
 = \int_{u/4}^{\infty} r^d \psi_u(r) \, F(dr)
 \le \left( \sup_{r > u/4} \psi_u(r) \right) \int_{u/4}^{\infty} r^d \, F(dr),
\end{equation}
where
\[
 \psi_u(r)
 = |B_{r_1/r+1}(o) \cap B_{r_2/r+1}(e_1 u/r)| - |B_1(o) \cap B_1(e_1 u/r)|.
\]
The equality in \eqref{eq:I2Integral} follows because $\psi_u(r) = 0$ when $u > 2(r_1 + r_2)$ and
$r < u/4$.

Note that by Lemma~\ref{lemma_reg_var}, the integral on the right side of~\eqref{eq:I2Integral} is
asymptotically equivalent to
\[
 \int_{u/4}^{\infty} r^d \, F(dr)
 \sim \left( \frac{\alpha}{\alpha - d} \right) 4^{\alpha-d} \bar F(u) u^d.
\]
In light of~\eqref{eq:I2Integral}, it hence suffices to show that
\begin{equation}
 \label{eq:PsiLimit}
 \sup_{r > u/4} \psi_u(r) \to 0
\end{equation}
as $u \to \infty$. This will be done by inspecting the geometry of $\psi_u$. Because the
intersection of the unit-balls above is a subset of the intersection of the larger balls, we can
bound the nonnegative function $\psi_u$ using the annuli around the unit-balls, so that
\begin{align*}
 \psi_u(r)
 &=   \left| \big( B_{r_1/r+1}(o) \cap B_{r_2/r+1}(e_1 u/r) \big) \setminus \big( B_1(o) \cap B_1(e_1 u/r) \big) \right| \\
 &\le \left| B_{r_1/r+1}(o) \setminus B_1(o) \right| + \left| B_{r_2/r+1}(e_1 u/r) \setminus B_1(e_1 u/r) \right| \\
 &= |B_1| ((r_1/r+1)^d - 1) + |B_1| ((r_2/r+1)^d - 1) \\
 &\le 2 |B_1| \left( \left( \frac{r_1 \vee r_2}{r} + 1 \right)^d - 1 \right).
\end{align*}
Because this bound is valid for all $u$, we conclude~\eqref{eq:PsiLimit}, and the proof is
complete.
\end{proof}

\begin{lemma}
\label{lemma_int}
Let $F$ be a probability distribution on $\Rp$ which follows a power law with tail exponent
$\alpha > d$. Then there exist constants $u > 0$ and $c > 0$ such that
\begin{equation}
 \label{eq:lemma_int}
 \int_\Rp |B_{r_1+r}(o) \cap B_{r_2+r}(z)| \, F(dr) \le c \, \bar F(|z|) |z|^d
\end{equation}
whenever $|z| > u$ and $r_1+r_2 \le |z|/2$.
\end{lemma}
\begin{proof}
Observe first using Lemma~\ref{lemma_reg_var} that
\[
 \int_{|z|/4}^{\infty} r^d \, F(dr) \sim c_1 |z|^d \bar F(|z|),
\]
where $c_1 = 4^{\alpha-d} \alpha/(\alpha-d)$. Hence, we may fix a constant $u>0$ such that
\begin{equation}
 \label{eq:lemma_int2}
 \int_{|z|/4}^{\infty} r^d \, F(dr) \le 2 c_1 |z|^d \bar F(|z|)
\end{equation}
whenever $|z| > u$.

Assume now that $|z| > u$ and $r_1+r_2 \le |z|/2$. In this case the intersection on the left side of \eqref{eq:lemma_int}
is nonempty only when $r > |z|/4$. For any such $r > |z|/4$, a crude estimate shows that
\[
 |B_{r_1+r}(o) \cap B_{r_2+r}(z)|
 \le |B_{r_1+r}(o)|
 \le |B_1| ( r_1 + r_2 + r)^d,
\]
which together with the inequality $r_1 + r_2 \le |z|/2 < 2r$
shows that
\[
 |B_{r_1+r}(o) \cap B_{r_2+r}(z)|
 \le 3^d |B_1| r^d.
\]
As a consequence,
\[
 \int_\Rp |B_{r_1+r}(o) \cap B_{r_2+r}(z)| \, F(dr)
 \le 3^d |B_1| \int_{|z|/4}^{\infty} r^d F(dr),
\]
so that by virtue of~\eqref{eq:lemma_int2}, the claim holds for $c = 2 c_1 3^d |B_1|$.
\end{proof}

\begin{lemma}
\label{lemma_s1}
Let $F$ be a probability distribution on $\Rp$ which follows a power law with tail exponent
$\alpha > d$, and let $h$ be an arbitrary positive function. Then there exist constants $u > 0$
and $c > 0$ such that
\[
 (|z|^d \bar F(|z|))^{-1} \int_\Rp |B_r(o) \cap B_r(z)| h(r) F(dr)
 \le c \sup_{r \ge |z|/2} h(r)
\]
for $|z| > u$.
\begin{proof}
By using Lemma~\ref{lemma_int}, fixing the constants $u$ and $c$ as in the lemma, and noting that
the integrand vanishes for $r \le |z|/2$, we see that
\[
 \int_\Rp |B_r(o) \cap B_r(z)| h(r) F(dr)
 \le h^*(z) \int_{\R+} |B_r(o) \cap B_r(z)| F(dr)
 \le c h^*(z) \bar F(|z|) |z|^d
\]
for all $|z| > u$, where $h^*(z) = \sup_{r \geq |z|/2} h(r)$.
\end{proof}
\end{lemma}

{\small
\subsection*{Acknowledgements}
We thank Volker Schmidt for suggesting to study this type of problems a long time ago back in
2005. We thank Antti Penttinen for inspiring discussions and valuable comments.}

\bibliographystyle{alpha}
\bibliography{lahteet-v1.3}

\begin{thebibliography}{MVDE02}

\bibitem[AHM06]{Andersson_Haggstrom_Mansson_2006}
Jenny Andersson, Olle H{\"a}ggstr{\"o}m, and Marianne M{\aa}nsson.
\newblock The volume fraction of a non-overlapping germ-grain model.
\newblock {\em Electron. Comm. Probab.}, 11:78--88, 2006.

\bibitem[BB09]{Baccelli_Blaszczyszyn_2009b}
F.~Baccelli and B.~B{\l}aszczyszyn.
\newblock {\em Stochastic Geometry and Wireless Networks, Volume II ---
  Applications}.
\newblock NoW, 2009.
\newblock http://hal.inria.fr/inria-00403040.

\bibitem[BGT87]{Bingham_Goldie_Teugels_1987}
Nicholas~H. Bingham, Charles~M. Goldie, and Jozef~L. Teugels.
\newblock {\em Regular Variation}.
\newblock Cambridge University Press, 1987.

\bibitem[BS03]{Bohm_Schmidt_2003}
Stephan B{\"o}hm and Volker Schmidt.
\newblock Palm representation and approximation of the covariance of random
  closed sets.
\newblock {\em Adv. Appl. Probab.}, 35(2):295--302, 2003.

\bibitem[CSN09]{Clauset_Shalizi_Newman_2009}
Aaron Clauset, Cosma~Rohilla Shalizi, and M.~E.~J. Newman.
\newblock Power-law distributions in empirical data.
\newblock {\em SIAM Rev.}, 51(4):661--703, 2009.

\bibitem[Dal99]{Daley_1999}
D.~J. Daley.
\newblock The {H}urst index of long-range dependent renewal processes.
\newblock {\em Ann. Probab.}, 27(4):2035--2041, 1999.

\bibitem[DEKS11]{Demichel_Estrade_Kratz_Samorodnitsky_2011}
Yann Demichel, Anne Estrade, Marie Kratz, and Gennady Samorodnitsky.
\newblock How fast can the chord length distribution decay?
\newblock {\em Adv. Appl. Probab.}, 43(2):504--523, 2011.

\bibitem[DV97]{Daley_Vesilo_1997}
D.~J. Daley and Rein Vesilo.
\newblock Long range dependence of point processes, with queueing examples.
\newblock {\em Stoch. Proc. Appl.}, 70(2):265--282, 1997.

\bibitem[DVJ08]{Daley_Vere-Jones_2008}
D.~J. Daley and D.~Vere-Jones.
\newblock {\em An Introduction to the Theory of Point Processes. {V}ol. {II}}.
\newblock Springer, New York, second edition, 2008.

\bibitem[Hae11]{Haenggi_2011}
Martin Haenggi.
\newblock Mean interference in hard-core wireless networks.
\newblock {\em IEEE Commun. Lett.}, 15(8):792--794, 2011.

\bibitem[JMST05]{Jones_Martinez_Saar_Trimble_2005}
Bernard J.~T. Jones, Vicent~J. Mart\'{\i}nez, Enn Saar, and Virginia Trimble.
\newblock Scaling laws in the distribution of galaxies.
\newblock {\em Rev. Mod. Phys.}, 76:1211--1266, 2005.

\bibitem[KLNS07]{Kaj_Leskela_Norros_Schmidt_2007}
Ingemar Kaj, Lasse Leskel{\"a}, Ilkka Norros, and Volker Schmidt.
\newblock Scaling limits for random fields with long-range dependence.
\newblock {\em Ann. Probab.}, 35(2):528--550, 2007.

\bibitem[KS01]{Kulik_Szekli_2001}
Rafa{\l} Kulik and Ryszard Szekli.
\newblock Sufficient conditions for long-range count dependence of stationary
  point processes on the real line.
\newblock {\em J. Appl. Probab.}, 38(2):570--581, 2001.

\bibitem[Mat60]{Matern_1960}
Bertil Mat{\'e}rn.
\newblock {\em Spatial variation: {S}tochastic models and their application to
  some problems in forest surveys and other sampling investigations}.
\newblock Meddelanden Fran Statens Skogsforskningsinstitut, Stockholm, 1960.

\bibitem[Mol05]{Molchanov_2005}
Ilya Molchanov.
\newblock {\em Theory of Random Sets}.
\newblock Springer-Verlag, London, 2005.

\bibitem[MR02]{Mansson_Rudemo_2002}
Marianne M{\aa}nsson and Mats Rudemo.
\newblock Random patterns of nonoverlapping convex grains.
\newblock {\em Adv. Appl. Probab.}, 34(4):718--738, 2002.

\bibitem[MVDE02]{Snethlage_Martinez_Stoyan_Saar_2002}
{M. Snethlage}, {V. J. Mart\'{\i}nez}, {D. Stoyan}, and {E. Saar}.
\newblock Point field models for the galaxy point pattern.
\newblock {\em Astron. Astrophys.}, 388(3):758--765, 2002.

\bibitem[NB12]{Nguyen_Baccelli}
Tien~Viet Nguyen and Fran{\c{c}}ois Baccelli.
\newblock Generating functionals of random packing point processes: {F}rom
  hard-core to carrier sensing.
\newblock \url{http://arXiv.org/abs/1202.0225}, 2012.

\bibitem[OM00]{Ohser_Mucklich_2000}
Joachim Ohser and Frank M{\"u}cklich.
\newblock {\em Statistical Analysis of Microstructures in Materials Science}.
\newblock Wiley, 2000.

\bibitem[Sam06]{Samorodnitsky_2006}
Gennady Samorodnitsky.
\newblock Long range dependence.
\newblock {\em Found. Trends Stoch. Syst.}, 1(3):163--257, 2006.

\bibitem[SKM95]{Stoyan_Kendall_Mecke_1995}
Dietrich Stoyan, Wilfrid~S. Kendall, and Joseph Mecke.
\newblock {\em Stochastic Geometry and its Applications}.
\newblock Wiley, second edition, 1995.

\bibitem[SSW02]{Schuth_Sing_Weitkamp_2002}
Ferdi Schüth, Kenneth S.~W. Sing, and Jens Weitkamp, editors.
\newblock {\em Handbook of Porous Solids}.
\newblock Wiley-VCH, 2002.

\bibitem[SW08]{Schneider_Weil_2008}
Rolf Schneider and Wolfgang Weil.
\newblock {\em Stochastic and Integral Geometry}.
\newblock Probability and its Applications (New York). Springer-Verlag, Berlin,
  2008.

\bibitem[VA98]{Vamvakos_Anantharam_1998}
Socrates Vamvakos and Venkat Anantharam.
\newblock On the departure process of a leaky bucket system with long-range
  dependent input traffic.
\newblock {\em Queueing Syst.}, 28(1--3):191--214, 1998.

\end{thebibliography}
\end{document}